\documentclass[11pt,reqno]{amsart}

\usepackage[dvipsnames]{xcolor}
\usepackage{amsfonts}
\usepackage{amsmath,amscd}
\usepackage{amssymb}
\usepackage{tikz-cd}
\usepackage{amsthm}
\usepackage{cases}

\usepackage{galois}
\usepackage{chemarrow}
\usepackage[all]{xy}
\usepackage{graphicx}
\usepackage[colorlinks,
            linkcolor=magenta,
            anchorcolor=black,
            citecolor=red
            ]{hyperref}	
\usepackage{mathrsfs}
\usepackage{extarrows}
\usepackage{texnames}
\usepackage{textcase}

\usepackage[deletedmarkup=xout]{changes}
\definechangesauthor[name={Per cusse}, color=OrangeRed]{.}

\makeatletter
\def\section{%
  \@startsection{section}{1}
    {\z@}
    {2.0ex plus 0.8ex minus .1ex}
    {1.0ex plus .2ex}
    {\large\bfseries\boldmath\centering\MakeTextUppercase}%
}
\makeatother

\newtheorem{thm}{Theorem}[section]
\newtheorem{lem}[thm]{Lemma}

\newtheorem{rem}[thm]{Remark}
\newtheorem{pro}[thm]{Proposition}
\newtheorem{defi}[thm]{Definition}

\newcommand*{\QEDB}{\hfill\ensuremath{\square}}
\numberwithin{equation}{section}

\newcommand\relphantom[1]{\mathrel{\phantom{#1}}}
\newcommand\repl{\relphantom}
\usepackage[utf8]{inputenc}

\begin{document}

\title[\sc\small S\MakeLowercase{tability of the generalized} L\MakeLowercase{agrangian mean curvature flow in cotangent bundle}]{\sc\LARGE S\MakeLowercase{tability of the generalized} L\MakeLowercase{agrangian mean curvature flow in cotangent bundle}}
\author[\sc\small X\MakeLowercase{ishen} J\MakeLowercase{in and} J\MakeLowercase{iawei} L\MakeLowercase{iu}]{\sc\large X\MakeLowercase{ishen} J\MakeLowercase{in and} J\MakeLowercase{iawei} L\MakeLowercase{iu}}
\address{Xishen Jin\\ Department of Mathematics\\Renmin University of China\\Beijing, 100872, China.} \email{jinxishen@ruc.edu.cn}
\address{Jiawei Liu\\ School of Mathematics and Statistics\\ Nanjing University of Science \& Technology\\ Xiaolingwei Street 200\\ Nanjing 210094\\ China.} \email{jiawei.liu@njust.edu.cn}
\thanks{X.S. Jin is supported by NSFC (Grant No.12001532). J.W. Liu is supported by NSFC (Grant No.12371059), Jiangsu Specially-Appointed Professor Program, Fundamental Research Funds for the Central Universities and Special Priority Program SPP 2026 ``Geometry at Infinity" from the German Research Foundation (DFG)}
\maketitle
\vskip -3.99ex

\centerline{\noindent\mbox{\rule{3.99cm}{0.5pt}}}

\vskip 5.01ex

\ \ \ \ {\bf Abstract.} In this paper, we consider the stability of the generalized Lagrangian mean curvature flow of graph case in the cotangent bundle, which is first defined by Smoczyk-Tsui-Wang \cite{STW}. By new estimates of derivatives along the flow, we weaken the initial condition and remove the positive curvature condition in \cite{STW}. More precisely, we prove that if the graph induced by a closed $1$-form is a special Lagrangian submanifold in the cotangent bundle of a Riemannian manifold, then the generalized Lagrangian mean curvature flow is stable near it.


\vspace{7mm}

\section{Introduction}
Special Lagrangian submanifolds were introduced by Harvey-Lawson \cite{HL} in their study of calibration geometry. They attract a lot of attention due to their relations to the  Strominger-Yau-Zaslow conjecture \cite{SYZ} on the mirror symmetry between Calabi-Yau manifolds. Since the special Lagrangian submanifolds are always minimal, a natural approach to find such submanifold is to evolve a Lagrangian submanifold along the mean curvature flow, which is the negative gradient flow of the area functional. When the ambient manifold is a K\"ahler-Einstein manifold, Smoczyk \cite{S96} proved that the Lagrangian property is preserved along the mean curvature flow. But this property no longer holds if the ambient manifold is a general symplectic manifold. Therefore, the original mean curvature flow can not be directly used in the study of the important conjectures which concern the Lagrangian isotopy problem in general symplectic manifolds, such as the cotangent bundles which do not carry K\"ahler-Einstein structures. Thereby how to find special Lagrangian submanifolds by providing Lagrangian deformation through geometric flow becomes an important problem. In \cite{B11}, Behrndt introduced the generalized mean curvature flow which evolves in direction of the generalized mean curvature vector field. He proved that the Lagrangian condition is kept along the flow in a K\"ahler manifold carrying an almost Einstein structure. In \cite{SW2011}, Smoczyk-Wang defined the generalized mean curvature flow in more general almost K\"ahler manifolds. Then Smoczyk-Tsui-Wang \cite{STW} comprehensively studied this flow in the cotangent bundle and proved a stability theorem near the zero section of the cotangent bundle.

An almost K\"ahler structure on a symplectic manifold $(M,\omega)$ consists a Riemannian metric $G$ and an almost complex structure $J$ such that the symplectic form $\omega$ satisfies that
\begin{equation}\label{0825000}
\omega(\cdot,\cdot)=G(J\cdot,\cdot).
\end{equation}
A symplectic manifold with an almost K\"ahler structure $(M,\omega,G,J)$ is referred as an almost K\"ahler manifold. To find special Lagrangian submanifolds in such manifold, Smoczyk-Wang \cite{SW2011} introduced the generalized mean curvature flow. A smooth family of almost Lagrangian immersions
\begin{equation}\label{0825004}
F:\Sigma\times[0,T)\to M
\end{equation}
is said to satisfy the generalized mean curvature flow if $F$ satisfies
\begin{equation}\label{0825005}
\frac{\partial F}{\partial t}(x,t)=\widehat H(x,t)\ \ \ \text{and}\ \ \ F(\Sigma,0)=\Sigma_0,
\end{equation}
where $\widehat H(x,t)$ is the generalized mean curvature vector with respect to the canonical connection on the almost Lagrangian submanifold $\Sigma_t=F(\Sigma,t)$ at $F(x,t)$. Since the cotangent bundle $T^*X$ of a Riemannian manifold $(X,g)$ admits a canonical almost K\"ahler structure with respect to the base metric $g$ (see section \ref{sec-2} or \cite{YI73,IV87,STW} for more details), in \cite{STW}, Smoczyk-Tsui-Wang studied the generalized mean curvature flow \eqref{0825005} in $T^*X$. They first proved that the Ricci from of the canonical connection on $T^*X$ vanishes (Theorem $1$ in \cite{STW}). Combining this with Theorem $2$ in \cite{SW2011}, they concluded that  the Lagrangian condition is preserved along the generalized mean curvature flow. Moreover, similar to Calabi-Yau case, they related the generalized mean curvature vector field of a Lagrangian submanifold to the Lagrangian angle through a holomorphic $n$-form. Then they proved that the generalized Lagrangian mean curvature flow of compact Lagrangians in the cotangent bundle keeps the exactness and the zero Maslov class (Theorem $2$ in \cite{STW}). 

Although the generalized Lagrangian mean curvature flow in cotangent bundle maintains many good properties, as we know from Neves's work \cite{AN1, AN2}, there are still many analytic difficulties even in the original Lagrangian mean curvature flow case. As the  pioneering step toward understanding this flow, Smoczyk-Tsui-Wang \cite{STW} considered the generalized Lagrangian mean curvature flow given as graphs of closed $1$-forms in the cotangent bundle of a Riemannian manifold $(X,g)$. Let $u_0$ be a smooth function on $X$. Then $du_0$, as a closed 1-form, induces a Lagrangian submanifold $\Sigma_0$ in $T^*X$. Starting from $\Sigma_0$, since this flow keeps the exactness, the generalized mean curvature flow $\Sigma_{t}$ is given as the graph of $du_t$ and $u_t$ satisfies the following fully nonlinear parabolic equation
\begin{equation}\label{GLMCF0}
  \frac{\partial }{\partial t}u_t=\theta(du_t)=\frac{1}{\sqrt{-1}}\log \frac{\det (g_{i j}+\sqrt{-1}(u_t)_{ij})}{\sqrt{\det g_{ij}}\sqrt{\det (\widetilde{\eta}_{t})_{ij}}}
\end{equation}
with initial value $u_0$, where $\theta(du_t)$ is the Lagrangian angle of $\Sigma_t$, and $(\widetilde\eta_t)_{ij}=g_{ij}+(u_t)_{ki}g^{kl}(u_t)_{lj}$ is the induced metric on $\Sigma_t$. For equation \eqref{GLMCF0}, Smoczyk-Tsui-Wang proved the following theorem.
\begin{thm}[Theorem $3$ in \cite{STW}]\label{0825001}
When $(X,g)$ is a standard round sphere of constant sectional curvature, then the zero section in $T^*X$ is stable under the generalized Lagrangian mean curvature flow \eqref{GLMCF0}.
\end{thm}
\begin{rem}\label{0825002}
The stability in Theorem \ref{0825001} is in the following sense. Suppose that a Lagrangian submanifold $\Sigma_0$ is the graph of $du_0$ for a smooth function $u_0$ on $X$ and let $\lambda_i$ be the eigenvalues of the Hessian of $u_0$ with respect to $g$. Then there exists a positive constant $\delta$ depending only on $n$ and the curvature of $X$ such that if $\prod_{i=1}^n (1+\lambda_i^2)\leqslant 1+\delta$, then the generalized mean curvature flow starting from $\Sigma_0$ exists smoothly for all time, and converges to the zero section smoothly at infinity.

In \cite{STW}, they also pointed that this stability theorem still holds true when the sphere is replaced by a compact Riemannian manifold of positive sectional curvature.
\end{rem}
There are some other stability results on the Lagrangian mean curvature flows. For examples, Smoczyk-Wang \cite{SW2002}, Zhang \cite{XWZ}, Chen-Pang \cite{CP09}, Chau-Chen-He \cite{CCH},  Chau-Chen-Yuan \cite{CCY} and Smoczyk-Tsui-Wang \cite{STW1} studied the stability of Lagrangian mean curvature flows when the base metrics are flat. Han-Jin \cite{HJ1} considered the stability of mean curvature flows on holomorphic line bundles, which can be seen as a complex version of the Lagrangian mean curvature flow. In \cite{TW1}, Tsui-Wang studied the stability of mean curvature flows in several well-known model spaces of manifolds. In fact, they proved the stability of mean curvature flows near the zero sections in the cotangent bundles of sphere with Stenzel metric and complex projective space with Calabi-Yau metric, and in the vector bundles over certain Einstein manifolds. There are also some results on the stability of higher co-dimensional mean curvature flows, see \cite{W1,W2, TW2, MW} for more details.

In this paper, by introducing new techniques for derivative estimates (especially for $C^2$-estimate), we proved that if the graph induced by a closed $1$-form is a special Lagrangian submanifold in the cotangent bundle, then the generalized Lagrangian mean curvature flow is stable near it. The innovations are that we remove the assumption on the positive curvature condition of the ambient manifold in \cite{STW} and that we generalize the initial condition from close to the zero section to the more general case close to the special Lagrangian submanifold. Moreover, by proving a Harnack-type inequality, we also deduce the exponentially convergence in smooth sense.

The main result in this paper is the following stability theorem.
\begin{thm}\label{0825003}
Let $(X,g)$ be a compact $n$-dimensional Riemannian manifold and $\hat\chi$ be a closed $1$-form on $X$. If the graph $\Sigma_{\hat \chi}$ induced by $\hat\chi$ is a special Lagrangian submanifold in $T^*X$. Then the generalized Lagrangian mean curvature flow is stable near $\Sigma_{\hat \chi}$. More precisely, there exists a positive constant $\delta_0$ depending only on $n$, $g$ and $\hat \chi$ such that for any $u_0\in C^\infty(X)$ whose Hessian satisfies $|D^2 u_0|^2_g\leqslant \delta_0$, the generalized Lagrangian mean curvature flow starting from the graph $\Sigma_{\hat \chi_{0}}$ of $\hat\chi_{0}=\hat\chi+du_0$ exists smoothly for all time, and converges exponentially to $\Sigma_{\hat \chi}$ in $C^\infty$-sense.
\end{thm}

Recently, Lee-Tsai \cite{LT24} proved a similar stability result for the Lagrangian mean curvature flow, which states that a minimal Lagrangian is stable under the Lagrangian mean curvature flow in the K\"ahler-Einstein manifold of non-positive curvature.

Assume that the generalized mean curvature flow $\Sigma_{t\in [0,T)}$ is given as the graph of closed 1-form $\hat\chi_t$. Since the positive constant $\delta_0$ in Theorem \ref{0825003} should be chosen sufficiently small, the oscillation of the Lagrangian angle $\theta(\hat \chi_{0})$ of $\Sigma_{\hat\chi_0}$ is also small, which implies that there exists a smooth functions $u_t:X\to \mathbb{R}$ such that $\hat\chi_t=\hat\chi+du_t$ (Theorem \ref{0822045}). Then by using Smoczyk-Tsui-Wang's results (Proposition $5.2$ in \cite{STW}), we conclude that $u_t$ also satisfies equation
\begin{equation}\label{GLMCF1}
\frac{\partial u_t}{\partial t} =\theta(\hat\chi_{t})=\frac{1}{\sqrt{-1}}\log \frac{\det (g_{i j}+\sqrt{-1}(\hat\chi_{t})_{j,i})}{\sqrt{\det g_{ij}}\sqrt{\det (\eta_{t})_{ij}}}
\end{equation}
with initial value $u_0$, where $(\hat\chi_{t})_{j,i}$ is the covariant derivative of $\hat\chi_{t}$ with respect to $g$, and $(\eta_t)_{ij}=g_{ij}+(\hat\chi_{t})_{k,i}g^{kl}(\hat\chi_{t})_{l,j}$ is the induced metric on $\Sigma_{\hat\chi_{t}}$. Therefore, in the case of Theorem \ref{0825003}, the generalized Lagrangian mean curvature flow \eqref{0825005} can be written globally as equation \eqref{GLMCF1}. 

The key step in proving Theorem \ref{0825003} is to prove that  the generalized Lagrangian mean curvature flow \eqref{GLMCF1} keeps the smallness of $|D^2u_t|^2_g$. Only by proving this property can we deduce the $C^3$-estimate and then the high order estimates, long-time existence and convergence. In the proof of this property, the first difficulty is how to deal with terms that contains the curvature of $g$. In \cite{STW}, Smoczyk-Tsui-Wang proved this property along the flow \eqref{GLMCF0} under the positive curvature condition on $g$. More precisely, by choosing a normal coordinate system with respect to $g$ near a point to diagonalize the Hessian of $u_t$ with $(u_t)_{ij}=\lambda_i \delta_{ij}$, where $u_t$ is a solution of equation \eqref{GLMCF0}, they deduced the following excellent expressions (see Proposition $6.1$ and Proposition $6.2$ in \cite{STW})
\begin{equation}\label{0826001}
\frac{\partial}{\partial t}\vartheta-\widetilde\eta_t^{ij}\vartheta_{ij}=-2\sum_{i=1}^n\frac{\lambda_i^2}{1+\lambda_i^2}-2\sum_{p=1}^n\frac{1}{1+\lambda_p^2}\left(\sum_{l,i}R_{lppi}(u_t)_l(u_t)_i\right)
\end{equation}
and
\begin{equation}\label{0826002}
\begin{split}
\frac{\partial}{\partial t}\widetilde{\rho}-\widetilde\eta_t^{ij}\widetilde{\rho}_{ij}&=\sum_{p,q,k}\frac{-1+\lambda_p\lambda_q-\lambda_k(\lambda_p+\lambda_q)}{(1+\lambda_p^2)(1+\lambda_q^2)(1+\lambda_k^2)}(u_t)^2_{pqk}\\
&\ \ \ -\sum_{p<k}\frac{2(\lambda_p-\lambda_k)^2}{(1+\lambda_p^2)(1+\lambda_k^2)}R_{kppk}+\sum_{p,k,l}\frac{(\lambda_k-\lambda_p)}{(1+\lambda_p^2)(1+\lambda_k^2)}(u_t)_lR_{klpk,p},
\end{split}
\end{equation}
where 
\begin{equation}\label{0826003}
\vartheta=|D u_t|_g^2\ \ \ \text{and}\ \ \ \widetilde{\rho}=\frac{1}{2}\log\frac{\det\widetilde\eta_t}{\det g},
\end{equation}
$(\widetilde\eta_t)_{ij}=g_{ij}+(u_t)_{ki}g^{kl}(u_t)_{lj}$, the covariant derivatives are all with respect to $g$, $Rm$ is the curvature tensor of $g$ and $R_{kppk}$ is the sectional curvature of $g$. If $(X,g)$ is a standard round sphere of constant sectional curvature (or a compact Riemannian manifold of positive sectional curvature), the maximum principle implies that the smallness of $\widetilde \rho$ is preserved along the flow \eqref{GLMCF0} and so is $|D^2u_t|^2_g$. Then using this property, Smoczyk-Tsui-Wang proved $C^3$-estimate and obtained high order estimates by the standard parabolic Schauder estimates. The convergence follows from \eqref{0826001}.

In our case, to control the terms containing the curvature of $g$ in the evolution equation of $|D^2u_t|^2_g$, we use the evolutions of $\left(u_t-u_0(p)-\theta(\hat \chi)t\right)^2$ and $|D u_t|^2_g$. In fact, there exist negative terms in these evolutions, which can be used to cancel the terms containing the curvature of $g$ (see section \ref{sec-4.1} for more details).

The second difficulty is due to the appearance of $|D^2u_t|_g$ in the evolution equation of $|D^2u_t|^2_g$. This type term appears due to the $1$-form $\hat\chi$ contained in $\hat\chi_t$ and $\eta_t$ in equation \eqref{GLMCF1}. However, it does not appear in \cite{STW} since $\hat\chi=0$. In fact, $|D^2u_t|_g$ is a bad term because it is larger than $|D^2u_t|^2_g$ when it is small. To cancel it, we use the trick $(\eta_t)_{ij}-\hat\eta_{ij}$, and then the new terms brought by $\hat\eta_{ij}$ can be cancelled by the assumption that $\hat\chi$ induces a special Lagrangian submanifold in $T^*X$, that is, $\theta(\hat\chi)$ is a constant (see Lemma \ref{0824018} for more details).

In addition to the above problems, there is a difficulty when we consider the convergence. That is, we can not get the convergence directly by \eqref{0826001} as in \cite{STW} due to the lack of  assumption on the positivity of the curvature of $g$. Here, our approach is to prove a Harnack-type inequality coupled with the generalized mean curvature flow \eqref{GLMCF1}, and then we deduce the convergence in $C^\infty$-sense. Furthermore, the convergence can be improved to be exponentially fast.

The remaining of this paper is organized as follows. In section \ref{sec-2}, we first review the geometry on cotangent bundle and the generalized Lagrangian mean curvature flow of graph case in cotangent bundle. Then we prove the uniqueness of special Lagrangian submanifold of graph case in cotangent bundles, and we give a global expression of the generalized Lagrangian mean curvature flow of graph case in cotangent bundles under a certain assumption. In section \ref{sec-3}, we give some evolution equations along the generalized Lagrangian mean curvature flow. In section \ref{sec-4}, we show the smallness of the norm of $D^2u_t$ is preserved along the generalized mean curvature flow and then prove the long-time existence of the flow. At last, we prove a Harnack-type inequality coupled with the generalized mean curvature flow, and then we deduce the exponentially convergence in $C^\infty$-sense in section \ref{sec-5}.

\medskip

{\bf Acknowledgements.} J.W. Liu would like to thank his postdoctoral supervisors, Professor Pengfei Guan and Professor Miles Simon, for their careful guidance on geometric flow. Both authors would like to thank Professor Chung-Jun Tsai for sending us their preprint.

\section{Preliminaries}
\label{sec-2}
In this section, we review the geometry on cotangent bundle and the generalized Lagrangian mean curvature flow in cotangent bundle, most of them can be found in \cite{YI73,IV87,STW}.
\subsection{The geometry on cotangent bundle}
Let $(X,g)$ be an $n$-dimensional Riemannian manifold with a Riemannian metric $g$. Let $\{q^j\}_{i=1,\cdots,n}$ be a local coordinate system on $X$ and $D$ be the covariant derivative with respect to $g$, then there holds
\begin{equation}\label{0822001}
D_{\frac{\partial}{\partial q^i}}\frac{\partial}{\partial q^j}=\Gamma_{ij}^k \frac{\partial}{\partial q^k},
\end{equation}
where $\Gamma^{k}_{ij}$ is the Christoffel symbol of $g$. If we denote $R^{\repl{j}i}_{j\repl{i}kl}$ as the curvature tensor of $g$, then
\begin{equation}\label{0822002}
R\left(\frac{\partial}{\partial q^k},\frac{\partial}{\partial q^l}\right)\frac{\partial}{\partial q^j}=D_{\frac{\partial}{\partial q^k}}D_{\frac{\partial}{\partial q^l}} \frac{\partial}{\partial q^j}-D_{\frac{\partial}{\partial q^l}}D_{\frac{\partial}{\partial q^k}} \frac{\partial}{\partial q^j}=R^{\repl{j}i}_{j\repl{i}kl}\frac{\partial}{\partial q^i}.
\end{equation}
Therefore, $R^{\repl{j}i}_{j\repl{i}kl}$ can be expressed by the Christoffel symbols as
\begin{equation}\label{0822003}
R^{\repl{j}i}_{j\repl{i}kl}=\frac{\partial}{\partial q^k} \Gamma^{i}_{jl} -\frac{\partial}{\partial q^l} \Gamma^{i}_{jk}+ \Gamma^{i}_{pk}\Gamma^{p}_{jl}-\Gamma^{i}_{pl}\Gamma^{p}_{jk}.
\end{equation}
Let $M=T^*X$ be the cotangent bundle of $X$. Taking a coordinate system $\{q^i,p_i\}_{i=1,\cdots,n}$ on $M$ such that on overlapping charts with coordinate $\{\tilde q^i,\tilde p_i\}_{i=1,\cdots,n}$, we have
\begin{equation}\label{0822004}
\tilde p_i=\sum_{j=1}^{n}\frac{\partial q^j}{\partial \tilde q^i} p_j,\ \ \ i=1,\cdots,n
\end{equation}
and the canonical symplectic form on $M$ is given by
\begin{equation}\label{0822004}
\omega=\sum_{i=1}^n dq^i\wedge dp_i=-d\left(\sum_{i=1}^np_idq^i\right).
\end{equation}
Define 
\begin{equation}\label{0822005}
\theta_i=dp_i-\Gamma^k_{il}p_kdq^l\ \ \ \text{and}\ \ \ X_i=\frac{\partial}{\partial q^i} +\Gamma^{k}_{il}p_k\frac{\partial}{\partial p_l}, \ \ \ i=1,\cdots,n
\end{equation}
then $\left\{dq^i,\theta_i\right\}_{i=1,\cdots,n}$ forms a basis on $T^*M$, which is dual to the basis $\left\{X_i,\frac{\partial}{\partial p_i}\right\}_{i=1,\cdots,n}$ on $TM$. At this time, the bundle projection $\pi:M\to X$ satisfies
\begin{equation}\label{0822006}
d\pi\left(X_i\right)=\frac{\partial}{\partial q^i} \ \ \ \text{and}\ \ \ d\pi\left(\frac{\partial}{\partial p_i}\right)=0.
\end{equation}
Therefore, the connection $D$ generates two distributions $\mathcal{H}$, $\mathcal{V}$ in $TM$, which are called the horizontal distribution and the vertical distribution of $TM$ respectively.

\begin{pro}
\label{0822007}
Let $M=T^*X$ be the cotangent bundle of Riemannian manifold $(X,g)$. The horizontal distribution $\mathcal{H}$ of  $TM$ is spanned by $X^i$ and the vertical distribution $\mathcal{V}$ is spanned by $\frac{\partial}{\partial p_i}$. In terms of these bases, the Riemmannian metric $G$ on $M$ (or on the tangent bundle $TM$ of $M$) satisfies
\begin{equation}\label{0822008}
G\left(\frac{\partial }{\partial p_i},\frac{\partial }{\partial p_j}\right)=g^{ij},\ \ \ G\left(X^i,\frac{\partial }{\partial p_j}\right)=0\ \ \ \text{ and }\ \ \ G\left(X^i,X^j\right)=g^{ij},
\end{equation}
where $X^i=g^{ij}X_j$. In terms of $\{dq^i, \theta_i\}_{i=1,\cdots,n}$, this metric can be expressed as
\begin{equation}\label{0822009}
G(\cdot,\cdot)=g^{ij}\theta_i\otimes \theta_j +g_{ij} dq^i\otimes dq^j.
\end{equation}
The almost complex structure $J$ on $TM$ is defined by
\begin{equation}\label{0822010}
\omega(\cdot,\cdot)=G(J\cdot,\cdot),
\end{equation}
and it satisfies
\begin{equation}\label{0822011}
JX^i=\frac{\partial}{\partial p_i},\ \ \ J\frac{\partial}{\partial p_i}=-X^i\ \ \ \text{ and }\ \ \ Jdq^i=-g^{ij}\theta_j.
\end{equation}
\end{pro}

Next, we recall the canonical connection $\widehat \nabla$ (\cite{IV87,STW}), which is defined by
\begin{equation}\label{0822012}
\widehat \nabla X^i =-\Gamma^{i}_{jk} dq^j\otimes X^k\ \ \ \text{ and }\ \ \ \widehat \nabla \frac{\partial}{\partial p_i}=-\Gamma^i_{jk}dq^j \otimes \frac{\partial}{\partial p_k}.
\end{equation}
This connection preserves the horizontal distribution and the vertical distribution. Also $X^i$ and $\frac{\partial }{\partial p_i}$ are parallel in the fiber direction. Since $\widehat \nabla$ is compatible with the Riemannian metric $G$ and the almost complex structure $J$ on $TM$ (that is, $\widehat\nabla G=0$ and $\widehat\nabla J=0$), the Ricci form $\widehat \rho$ of $\widehat\nabla$ is given by
\begin{equation}\label{0822013}
\widehat \rho(V,W)=\frac{1}{2}\sum\limits_{\alpha=1}^{2n}G(\widehat R(V,W)Je_\alpha, e_\alpha)=\frac{1}{2}\sum\limits_{\alpha=1}^{2n}\omega(\widehat R(V,W)e_\alpha, e_\alpha),
\end{equation}
where $\widehat R$ is the curvature tensor of $\widehat \nabla$ and $\left\{  e_\alpha \right\}_{\alpha=1,\cdots,2n}$ is an arbitrary orthonormal basis of $TM$. 

In \cite{SW2011}, Smoczyk-Wang defined the following Einstein connection.
\begin{defi} \label{0822016}
A connection $\widehat\nabla$ which is compatible with the K\"ahler metric $G$ and the almost complex structure $J$ on an almost K\"ahler manifold $(M, \omega, G, J)$ is called Einstein, if the Ricci form $\widehat\rho$ of $\widehat\nabla$ satisfies
\begin{equation}\label{0822014}
\widehat \rho=f\omega
\end{equation}
for some smooth function $f$ on $M$.
\end{defi}
In \cite{STW}, Smoczyk-Tsui-Wang proved that $\widehat\nabla$ defined in \eqref{0822012} on $M$ is an Einstein connection with vanishing Ricci form. 
\begin{thm}[Theorem $1$ in \cite{STW}]\label{0822015}
Let $(X,g)$ be a Riemannian manifold and $(\omega,G,J)$ be an almost K\"ahler structure defined on the cotangent bundle $M=T^*X$ with the canonical connection $\widehat \nabla$. Then the Ricci form $\widehat \rho$ of $\widehat \nabla$ defined in \eqref{0822013} vanishes. In particular, $\widehat \nabla$ is an Einstein connection in the sense of Definition \ref{0822016}.
\end{thm}
Hence the cotangent bundle of a Riemannian manifold $(X,g)$ admits a naturally defined almost K\"ahler structure $(\omega,G,J)$ and a canonical connection $\widehat\nabla$ that is symplectic (that is, $\widehat\nabla\omega=0$), compatible with the K\"ahler metric $G$ and the almost complex structure $J$. Moreover, the Ricci form $\widehat\rho$ of $\widehat\nabla$ vanishes. 

\subsection{The generalized mean curvature flow in cotangent bundle}
\label{sec-2.2}
We denote the projections of $TM$ onto the horizontal distribution $\mathcal{H}$ and the vertical distribution $\mathcal{V}$ by $\pi_1$ and $\pi_2$ respectively. In terms of $\theta_i$ and $dq^i$, we have
\begin{equation}\label{0822017}
\pi_1=dq^i\otimes X_i \ \ \ \text{ and }\ \ \ \pi_2=\theta_i\otimes \frac{\partial }{\partial p_i}.
\end{equation}
Since $J$ interchanges $\mathcal{H}$ and $\mathcal{V}$, we get
\begin{equation}\label{0822018}
J\circ \pi_1=\pi_2 \circ J\ \ \ \text{ and }\ \ \ J\circ\pi_2=\pi_1 \circ J.
\end{equation}
With respect to these structures, Smoczyk-Tsui-Wang \cite{STW} defined an $n$-form $\Omega$ on $M$.
\begin{defi}
\label{0822019}
The $n$-form $\Omega$ is defined as
\begin{equation}\label{0822019}
\Omega=\sqrt{\det g_{ij}}(dq^1-\sqrt{-1}Jdq^1)\wedge \cdots \wedge (dq^n-\sqrt{-1}Jdq^n).
\end{equation}
It can be viewed as an $(n,0)$-form in the sense that
\begin{equation}\label{0822020}
\Omega(JV_1,V_2,\cdots,V_n)=\sqrt{-1}\Omega(V_1,\cdots,V_n).
\end{equation}
\end{defi}
\begin{pro}[Proposition $2.2$ in \cite{STW}]\label{0822021}
The $(n,0)$-form $\Omega$ on $M$ is parallel with respect to the connection $\widehat \nabla$.
\end{pro}
Now we recall the definition of the generalized mean curvature vector field of a Lagrangian immersion and relate it to the Lagrangian angle through the above holomorphic $n$-form $\Omega$. As in the classical theory, an immersion $F:\Sigma\to M=T^*X$ with dimension $n$ is Lagrangian if $F^*\omega|_\Sigma=0$. In this paper, we identify $\Sigma$ with the image of the Lagrangian immersion and refer $\Sigma$ as a Lagrangian submanifold if there is no confusion.

As described in \cite{STW} (see also \cite{SW2011} for the general case), the generalized mean curvature form on $\Sigma$ is defined by
\begin{equation}\label{0822022}
\nu_i= \sum_{k=1}^{n} \langle \widehat \nabla_{e_i}e_k, Je_k \rangle,\ \ \ i=1,\cdots,n
\end{equation}
where $\{e_i\}_{i=1,\cdots,n}$ is an orthonormal basis with respect to the induced metric on $\Sigma$ by the immersion $F$. The generalized mean curvature vector $\widehat H$ is defined as
\begin{equation}\label{0822023}
\widehat H=\sum_{i=1}^{n} \nu_i Je_i.
\end{equation}

A half dimensional subspace $L$ in $\mathbb{C}^n$ with the standard symplectic form is called Lagrangian if $L^\perp =L$. Given a Lagrangian subspace $L_0$ in $\mathbb{C}^n$, the Lagrangian angle of another Lagrangian subspace $L_1$ is defined by the argument of $\det U$, where $U$ is a unitary $n\times n$ matrix such that $L_1=UL_0$. More precisely, if we choose an orthogonal basis $e_1^i,\cdots,e_n^i$ for $L_i$ and set
\begin{equation}\label{0822024}
\nu_k^i=\frac{1}{\sqrt{2}}(e_k^i-\sqrt{-1}J e_k^i),\ \ \ \ i=0,1
\end{equation}
to be the associated holomorphic basis. Then the argument of $\det U_k^l$ with $\nu_k^1=U_{k}^l \nu_l^0$ is the Lagrangian angle of $L_1$ with respect to $L_0$. In \cite{STW}, Smoczyk-Tsui-Wang derived the following formula for the Lagrangian angle in terms of arbitrary basis.
\begin{lem}[Lemma $3.1$ in \cite{STW}]\label{0822028}
Suppose that $(V,\langle\cdot,\cdot\rangle)$ is an $2n$-dimensional (real) inner product space with a compatible almost complex structure $J$ (that is, $J$ is an isometry and $J^2=-\operatorname{\mathop{id}})$. Let $L_0$ be a fixed Lagrangian subspace of $V$ spanned by $\overline{v}_1,\cdots,\overline{v}_n$. Suppose that $L_1$ is another Lagrangian subspace spanned by $v_1,\cdots,v_n$ and that $v_i=\sum_{j=1}^n(\alpha_{ij}\overline{v}_j + \beta_{ij}J\overline{v}_j)$ for $i=1,\cdots,n$. Then the Lagrangian angle $\theta$ of $L_1$ with respect to $L_0$ is the argument of $\det(\alpha_{ij}+\sqrt{-1}\beta_{ij})$. In fact, they are related by
\begin{equation}\label{0822025}
\frac{\det(\alpha_{ij}+\sqrt{-1}\beta_{ij})\sqrt{\det \langle\overline{v}_i,\overline{v}_j \rangle}}{\sqrt{\det \langle{v}_i,{v}_j \rangle}} =e^{\sqrt{-1}\theta}.
\end{equation}
\end{lem}
\begin{rem}\label{0822026}
The Lagrangian angle is not uniquely defined, but it is defined up to adding an integer multiple of $2\pi$.
\end{rem}
Next, we consider the Lagrangian angle $\theta$ of the Lagrangian immersion $F:\Sigma\to M$ with respect to the horizontal distribution $\mathcal{H}$. Then we have the following proposition by using Lemma \ref{0822028}.
\begin{pro}[Proposition $3.1$ in \cite{STW}]\label{0822027}
Suppose that a Lagrangian submanifold $\Sigma$ of $M=T^*X$ is given by $F:\Sigma\to M$. Let $\{F_i\}_{i=1,\cdots,n}$ be an arbitrary basis tangential to $\Sigma$. Then the Lagrangian angle $\theta$ with respect to the horizontal distribution of $TM$ is
\begin{equation}\label{0822029}
\sqrt{-1}\theta=\log \det \left(G\left(F_i,X^j\right) +\sqrt{-1}G\left(F_i,\frac{\partial}{\partial p_j}\right)\right) +\frac{1}{2}\log \det g_{ij}- \frac{1}{2}\log \det \eta_{ij},
\end{equation}
where $\eta_{ij}=G\left( F_i,F_j\right)$.
\end{pro}

According to Proposition $3.2$ in \cite{STW}, the Lagrangian angle $\theta$ can be related to the $n$-form $\Omega$ by
\begin{equation}\label{0822030}
*(\Omega|_\Sigma)=e^{\sqrt{-1}\theta},
\end{equation}
where $*$ is the Hodge star operator on $\Sigma$ with respect to the induced metric $\eta_{ij}$ on $\Sigma$. The generalized mean curvature field $\widehat H$ of $\Sigma$ is given by
\begin{equation}\label{0822031}
\widehat H=J\nabla \theta,
\end{equation}
where $\nabla$ is the gradient operator on $\Sigma$ with respect to the induced metric. The special Lagrangian submanifold in the cotangent bundle of a Riemannian manifold is defined as
\begin{defi}\label{0822032}
A Lagrangian submanifold $\Sigma$ of $M=T^*X$ given by the immersion $F:\Sigma\to M$ is called a special Lagrangian submanifold if and only if its Lagrangian angle $\theta$ with respect to the horizontal distribution $\mathcal{H}$ of $TM$ is a constant.
\end{defi}

To find the special Lagrangian submanifold, Smoczyk-Wang \cite{SW2011} introduced the generalized mean curvature flow. A smooth family of almost Lagrangian immersions
\begin{equation}\label{0822033}
F:\Sigma\times [0,T)\to M
\end{equation}
is said to satisfy the generalized mean curvature flow if $F$ satisfies
\begin{equation}\label{0822034}
\frac{\partial F}{\partial t}(x,t)=\widehat H(x,t)\ \ \ \text{ and }\ \ \ F(\Sigma,0)=\Sigma_0,
\end{equation}
where $\widehat H(x,t)$ is the generalized mean curvature vector with respect to the canonical connection on the almost Lagrangian submanifold $\Sigma_t=F(\Sigma,t)$ at $F(x,t)$. Recall that a submanifold $\Sigma$ is said to be almost Lagrangian if it satisfies $J(T\Sigma)\cap T\Sigma=\{0\}$. In \cite{SW2011}, Smoczyk-Wang proved that the Lagrangian condition is preserved by this flow.
\begin{thm}[Theorem $2$ in \cite{SW2011}]\label{0822035}
Suppose that $(M,\omega,G,J)$ is an almost K\"ahler manifold and $\widetilde\nabla$ is an Einstein connection which is compatible with $G$ and $J$. If $\Sigma_0$ is a closed Lagrangian submanifold of $M$. Then the generalized mean curvature flow \eqref{0822034}
with respect to $\widetilde\nabla$ preserves the Lagrangian condition.
\end{thm}
\begin{rem}\label{0904001}
From Theorem \ref{0822015} and Theorem \ref{0822035}, the generalized mean curvature flow defined by the generalized mean curvature vector $\widehat H$ with respect to $\widehat\nabla$ in the cotangent bundle $M$ of a Riemannian manifold $X$ preserves the Lagrangian condition if the initial one is a closed Lagrangian submanifold of $M$.
\end{rem}

\subsection{The generalized mean curvature flow in the graph case}
As a ``test case" for the more general non-graph situation, it is interesting to consider the generalized Lagrangian mean curvature flow of Lagrangian graphs that are induced by 1-forms on $X$ in the cotangent bundle $M$ of a Riemannian manifold $(X,g)$. Let $\chi$ be a smooth $1$-form on $X$. Define
\begin{equation}\label{0822036}
F_{\chi}:X\to M=T^*X,\ \ \ \text{   }\ \ \ F_{\chi}(x)=(x,\chi(x)).
\end{equation}
Then the graph $\Sigma_\chi=(x,\chi(x))\subseteq M$ of $\chi$ is Lagrangian if and only if $\chi$ is closed. In terms of the basis $\left\{X_i,\frac{\partial}{\partial p_i}\right\}$ of $TM$, the tangent space of $\Sigma_\chi$ is spanned by the basis
\begin{equation}\label{0822037}
(F_{\chi})_{i}:= \frac{\partial F_\chi}{\partial q^i}=X_i +\chi_{j,i}\frac{\partial}{\partial p_j},
\end{equation}
where $\chi_{j,i}$ denotes the covariant derivative of the $1$-form $\chi$ with respect to the metric $g$ on $X$. 

\begin{rem}\label{0904002}If a Lagrangian submanifold is given as the graph of a closed $1$-form $\chi$ in the cotangent bundle, then its exactness is equivalent to the exactness of $\chi$. In fact, on $X$, 
\begin{equation}\label{08220371}
F_{\chi}^*\lambda(\frac{\partial}{\partial q^i})=\lambda((F_\chi)_*\frac{\partial}{\partial q^i})=\lambda(\frac{\partial F_\chi}{\partial q^i})=\lambda(X_i +\chi_{j,i}\frac{\partial}{\partial p_j})=p_i=\chi(\frac{\partial}{\partial q^i}),
\end{equation}
where $\lambda=\sum\limits_{i=1}^np_idq^i$ is the Liouville form. If the Lagrangian submanifold is exact, that is, $F_{\chi}^*\lambda$ is exact. Then there is a smooth function $f$ on $X$ such that
\begin{equation}\label{08220371}
F_{\chi}^*\lambda(\frac{\partial}{\partial q^i})=df(\frac{\partial}{\partial q^i}),
\end{equation}
which implies that $\chi=df$ on $X$.
\end{rem}

According to Proposition \ref{0822027}, the Lagrangian angle of $\Sigma_\chi$ can be expressed as follows.

\begin{pro}[Proposition 5.1 in \cite{STW}]\label{0822039}
Suppose that $\Sigma_\chi$ is a Lagrangian submanifold of $M=T^*X$ defined as the graph of a closed $1$-form $\chi$ on $X$. Then the Lagrangian angle $\theta(\chi)$ of $\Sigma_\chi$ with respect to the horizontal distribution $\mathcal{H}$ of $TM$ is
\begin{equation}\label{0822038}
e^{\sqrt{-1}\theta(\chi)}=\frac{\det (g_{i j}+\sqrt{-1}\chi_{j,i})}{\sqrt{\det g_{ij}}\sqrt{\det \eta_{ij}}},
\end{equation}
where $g_{ij}$ is the metric on $X$, $\chi_{j,i}$ is the covariant derivative of $\chi$ with respect to $g$ and $\eta_{ij}=G\left((F_\chi)_i,(F_\chi)_j\right)=g_{ij}+\chi_{k,i}g^{kl}\chi_{l,j}$ is the induced metric on $\Sigma_\chi$.
\end{pro}
By using formula \eqref{0822038}, the derivative of $\theta(\chi)$ is given by
\begin{lem}[Lemma $6.1$ in \cite{STW}]\label{0822040}
The derivative of $\theta(\chi)$ is given by
\begin{equation}\label{0822041}
\left(\theta(\chi)\right)_{k}=\eta^{pq}\chi_{p,qk}.
\end{equation}
\end{lem}
In the graph case, we call the graph $\Sigma_\chi$ of $\chi$ in the cotangent bundle $M$ a special Lagrangian submanifold if and only if $\theta$ is a constant due to Definition \ref{0822032}. According to Lemma \ref{0822040}, if $\Sigma_\chi$ is a special Lagrangian submanifold, then
\begin{equation}\label{0822042}
\eta^{pq}\chi_{p,qk}=0.
\end{equation}
Differentiating this equation one more time, we get
\begin{equation}
\label{0822043}
\eta^{pq}_{\repl{pq},l}\chi_{p,qk}+ \eta^{pq}\chi_{p,qkl}=0.
\end{equation}
From Proposition $5.2$ in \cite{STW}, the generalized Lagrangian mean curvature flow of graph case in the cotangent bundle can be expressed locally as a fully nonlinear parabolic equation for the locally defined potential function $u_t$ (with $\chi_t=du_t$) on $X$,
\begin{equation}\label{0822044}
\frac{\partial u_t}{\partial t} =\theta(\chi_t)=\frac{1}{\sqrt{-1}}\log \frac{\det (g_{i j}+\sqrt{-1}(u_t)_{ij})}{\sqrt{\det g_{ij}}\sqrt{\det (\eta_{t})_{ij}}},
\end{equation}
where $(\eta_t)_{ij}=g_{ij}+(u_t)_{ki}g^{kl}(u_t)_{lj}$. 
\begin{rem}\label{0823001}

Usually, the flow \eqref{0822044} is only defined locally. In \cite{STW}, Smoczyk-Tsui-Wang proved that if the initial Lagrangian submanifold is induced as the graph of an exact $1$-form, then this flow is defined globally. More precisely, let $u_0$ be a smooth function on $X$. If the generalized mean curvature flow is given as the graph of a closed $1$-form $\chi_t$ and the initial Lagrangian submanifold $\Sigma_0$ is indued as the graph of $du_0$, since the generalized Lagrangian mean curvature flow of compact Lagrangians in $T^*X$ keeps the exactness if $\Sigma_0$ is exact (Theorem $2$ in \cite{STW}), by Remark \ref{0904002}, there exists a smooth function $u_t$ such that $\chi_t=du_t$ and that $u_t$ is a global solution to flow \eqref{0822044}. But for more general initial Lagrangian submanifold which is indued only as the graph of a closed $1$-form $\chi_0$, the above arguments may not be valid.
\end{rem}

In this paper, we first prove that above flow is also globally defined when the oscillation of the Lagrangian angle of the initial Lagrangian submanifold is less than $2\pi$.
 
\begin{thm}\label{0822045}
Suppose that $\Sigma_{t\in [0,T)}$ is a generalized mean curvature flow given as the graph of a closed $1$-form $\chi_t$ and the Lagrangian angle $\theta(\chi_0)$ of $\Sigma_0$ satisfies
\begin{equation}\label{0822046}
\operatorname{\mathop{osc}}(\theta(\chi_0))<2\pi.
\end{equation}
Then the $1$-form $\chi_t$ is in the cohomology class $[\chi_0]\in H^1(X,\mathbb{R})$, that is, there exists smooth function $u_t$ on $X$ such that $\chi_t=\chi_0+du_t$. Furthermore, $u_t$ satisfies
\begin{equation}\label{0822047}
\frac{\partial u_t}{\partial t} =\theta(\chi_{t})=\frac{1}{\sqrt{-1}}\log \frac{\det (g_{i j}+\sqrt{-1}(\chi_{t})_{j,i})}{\sqrt{\det g_{ij}}\sqrt{\det (\eta_{t})_{ij}}},
\end{equation}
where $g_{ij}$ is the metric on $X$, $(\chi_{t})_{j,i}$ is the covariant derivative of $\chi_{t}$ with respect to $g$ and $(\eta_t)_{ij}=g_{ij}+(\chi_{t})_{k,i}g^{kl}(\chi_{t})_{l,j}$ is the induced metric on $\Sigma_{\chi_{t}}$.
\end{thm}
\begin{proof}
By using \eqref{0822041} and \eqref{0822044}, we have 
\begin{equation}\label{0904003}
\frac{\partial}{\partial t}\theta(\chi_t)=\eta_t^{ij}\left(\theta(\chi_t)\right)_{ij},
\end{equation}
where $\left(\theta(\chi_t)\right)_{ij}$ is the covariant derivative of $\theta(\chi_t)$ with respect to $g$, $(\eta_t)_{ij}=g_{ij}+(\chi_{t})_{k,i}g^{kl}(\chi_{t})_{l,j}$ is the induced metric on $\Sigma_{\chi_{t}}$ and $\eta_t^{ij}$ is the inverse of  $(\eta_t)_{ij}$. The maximum principle implies that 
\begin{equation}\label{0822048}
\operatorname{ \mathop{osc}}(\theta(\chi_t))<2\pi
\end{equation}
along the generalized mean curvature flow. Hence, $\theta_t=\theta(\Sigma_t)$ is a single value smooth function. By using equation \eqref{0822044} again, we obtain that
\begin{equation}\label{0822049}
\frac{\partial}{\partial t}\chi_t=d\theta_t.
\end{equation}
Hence, $\chi_t$ is in the cohomology class $[\chi_0]$ for all $t\in [0,T)$, and \eqref{0822047} follows from Proposition $5.2$ of \cite{STW} and Proposition \ref{0822039}.
\end{proof}

\subsection{The uniqueness of special Lagrangian submanifolds of graph case in the cotangent bundle}\label{subsec-unique}

In this subsection, we prove the uniqueness of special Lagrangian submanifolds in the cotangent bundle if they are induced as the graphs of $1$-forms.
\begin{thm}\label{0815100}
Let $(X,g)$ be a compact Riemannian manifold. If $\chi_0$ and $\chi_1$ are two closed $1$-forms in $[\chi]\in H^1(M,\mathbb{R})$ and both of them induce special Lagrangian submanifolds as their graphs, then $\chi_0=\chi_1$.
\end{thm}
\begin{proof}
Since $\chi_0$ and $\chi_1$ are in the same cohomology class $[\chi]$, there exists a smooth function $u$ such that
\begin{equation}\label{0822050}
\chi_1=\chi_0+du.
\end{equation}
Let $\chi_t=\chi_0 +tdu$ and $\theta(\chi_t)$ be the Lagrangian angle of $\Sigma_{\chi_t}$ that are induced as the graph of $\chi_t$. According to Lemma \ref{0822040}, we have
\begin{equation}\label{0822051}
\frac{d}{dt}\theta(\chi_t)=\eta^{pq}_t \frac{d}{dt}( \chi_t)_{p,q}=\eta^{pq}_t u_{pq},
\end{equation}
where $(\eta_t)_{pq}$ is the metric induced on $\Sigma_{\chi_t}$ and $\eta_t^{pq}$ is the inverse of $(\eta_t)_{pq}$. Integrating \eqref{0822051} from $0$ to $1$ on both sides, we get
\begin{equation}\label{0822052}
\theta(\chi_1)-\theta(\chi_0)=\int_0^1 \eta^{pq}_t dt\ u_{pq}=\tilde \eta^{pq}u_{pq},
\end{equation}
where $\tilde \eta^{pq}=\int_0^1 \eta^{pq}_t d t$ is a positive definite matrix. We assume $u(x_0)=\min\limits_{X} u(x)$. Then at point $x_0$, by \eqref{0822052} and the maximum principle, we have $\theta(\chi_1)\geqslant \theta(\chi_0)$ at $x_0$. Since $\theta(\chi_1)$ and $\theta(\chi_0)$ are constants,  $\theta(\chi_1)\geqslant \theta(\chi_0)$ on $X$. Similar arguments imply that $\theta(\chi_1)\leqslant \theta(\chi_0)$ also holds on $X$. Hence $\theta(\chi_1)= \theta(\chi_0)$ and then $u$ is a constant, that is, $ \chi_0=\chi_1$.
\end{proof}

\section{Evolution equations along the generalized Lagrangian mean curvature flows}
\label{sec-3}
In this section, we give the evolution equations of some geometric quantities along the generalized Lagrangian mean curvature flow \eqref{GLMCF1}. For convenience, we will write $u$ and $\eta$ instead of $u_t$ and $\eta_t$. We denote $\Delta_\eta=\eta^{ij}D_iD_j$, where $D$ is the covariant derivative with respect to $g$.

\begin{lem}\label{0822053}
The evolution equation of $(u-u_0(p)-\theta(\hat \chi)t)^2$ along the flow \eqref{GLMCF1} is given by
\begin{equation}\label{0822054}
(\frac{\partial }{\partial t} -\Delta_\eta ) (u-u_0(p)-\theta(\hat \chi)t)^2 =2(u-u_0(p)-\theta(\hat \chi) t)(\theta(\hat \chi_u)-\theta(\hat \chi) -\Delta_\eta u) -2 \eta^{ij}u_{i}u_{j},
\end{equation}
where $p\in X$ is a fixed point.
\end{lem}
\begin{proof}
Straightforward calculations show that
\begin{equation}\label{0822055}
\frac{\partial }{\partial t}(u-u_0(p)-\theta(\hat \chi)t)^2=2(u-u_0(p)-\theta(\hat \chi)t)(\theta(\hat \chi_u)-\theta(\hat \chi))
\end{equation}
and
\begin{equation}\label{0822056}
\Delta_{\eta}(u-u_0(p)-\theta(\hat \chi)t)^2=2(u-u_0(p)-\theta(\hat \chi)t)\Delta_{\eta} u +2\eta^{ij} u_iu_j.
\end{equation}
Then we get the desired equation.
\end{proof}

\begin{lem}\label{0823002}
The function $\vartheta=g^{ij}u_i u_j$ satisfies
\begin{equation}\label{0823003}
(\frac{\partial}{\partial t} -\Delta_\eta)\vartheta =-2\eta^{pq}g^{ij} u_{ip}u_{jq}-2 g^{ij}\eta^{pq} R_{p\repl{l}iq}^{\repl{p}l} u_{l}u_{j} +2g^{ij}\eta^{pq} \hat \chi_{p,qi}u_{j}.
\end{equation}
\end{lem}

\begin{proof}
Firstly, by using \eqref{0822047} and Lemma \ref{0822040}, we have
\begin{equation}\label{0823004}
\begin{split}
\frac{\partial}{\partial t} \vartheta&=\frac{\partial}{\partial t} (g^{i j} u_{i} u_{j})=2g^{i j}\left(\frac{\partial u}{\partial t}\right)_{i} u_{j}\\
&=2g^{i j}(\theta(\hat\chi_u))_{i} u_{j}=2g^{i j}\eta^{pq}(\hat\chi_{u})_{p,qi} u_{j}\\
&=2g^{i j}\eta^{p{q}}\hat \chi_{p,{q}i}u_{{j}}+2g^{i j}\eta^{p{q}}u_{pqi}u_{j}.
\end{split}
\end{equation}
Then direct computations show that
\begin{equation}\label{0823005}
\begin{split}
\Delta_{\eta}\vartheta&=\eta^{p{q}}(g^{ij}u_{i} u_{j})_{pq}\\
&=2\eta^{p{q}}g^{ij}(u_{ip}u_{jq}+u_{ipq}u_{j})\\
&=2\eta^{p{q}}g^{ij}(u_{ip}u_{jq}+u_{pqi}u_{j})+2\eta^{pq}g^{ij}R_{p\repl{l}iq}^{\repl{p}l} u_{j}u_{l}.
\end{split}
\end{equation}
Adding these equalities together gives us the required equation.
\end{proof}
\begin{lem}\label{0823006}
The function $\rho=g^{ij}g^{pq}u_{ip}u_{jq}$ satisfies the following evolution equation
\begin{equation}\label{0823007}
\begin{split}
(\frac{\partial}{\partial t} -\Delta_\eta) \rho&=-2g^{i j}g^{k l}\eta^{p b}\eta^{a q}\eta_{a b, l} (\hat\chi_u)_{p, qi}u_{k j} +2g^{i j}g^{k l}\eta^{p q} \hat \chi_{p, qil}u_{k j}\\
&\ \ \ -2g^{ij}g^{kl}\eta^{p q} u_{il p}u_{kjq}+2g^{i j}g^{k l}\eta^{p q} (\Xi_1)_{pqil}u_{k j},
\end{split}
\end{equation}
where
\begin{equation}\label{0823008}
(\Xi_1)_{pqil}=u_{kl}R^{\repl{p}k}_{p\repl{k}qi} +u_{k}R^{\repl{p}k}_{p\repl{k}qi,l} +u_{kp}R^{\repl{i}k}_{i\repl{k}ql}+u_{ik}R^{\repl{p}k}_{p\repl{k}ql} +u_{kq}R^{\repl{i}k}_{i\repl{k}pl} +u_{k}R^{\repl{i}k}_{i\repl{k}pl,q}.
\end{equation}
\end{lem}
\begin{proof}
By using \eqref{0822047}, Lemma \ref{0822040} and $\hat\chi_u=\hat\chi+du$, we have
\begin{equation}\label{0823009}
\begin{split}
\frac{\partial}{\partial t}\rho&=2g^{i j}g^{k l}\left(\frac{\partial u}{\partial t}\right)_{i l}u_{kj}=2g^{i j}g^{k l}(\theta(\hat\chi_u))_{i l}u_{k j}=2g^{i j}g^{k l}(\eta^{p q} (\hat\chi_u)_{p, qi})_{l}u_{k j}\\
&=-2g^{i j}g^{k l}\eta^{p b}\eta^{a q}\eta_{a b, l} (\hat\chi_u)_{p, qi}u_{k j} +2g^{i j}g^{k l}\eta^{p q} \hat \chi_{p, qil}u_{k j}\\
&\ \ \ +2g^{i j}g^{k l}\eta^{p q} u_{pqil}u_{k j}.\\
\end{split}
\end{equation}
Direct computations give
\begin{equation}\label{0823010}
\begin{split}
\Delta_{\eta}\rho&=\eta^{p q}(g^{i j}g^{k l}u_{i l}u_{k j})_{p q}\\
&=2g^{ij}g^{kl}\eta^{p q} u_{il p}u_{k j q}+2g^{i j}g^{k l}\eta^{p q}u_{k j}u_{i l p q}.
\end{split}
\end{equation}
By the communication formula for curvature tensor, we have
\begin{equation}\label{0823011}
\begin{split}
u_{pqil}&=(u_{pqi}-u_{piq})_l+(u_{ipql}-u_{iplq})+(u_{ipl}-u_{ilp})_q+u_{ilpq}\\
&=u_{kl}R^{\repl{p}k}_{p\repl{k}qi} +u_{k}R^{\repl{p}k}_{p\repl{k}qi,l} +u_{kp}R^{\repl{i}k}_{i\repl{k}ql}+u_{ik}R^{\repl{p}k}_{p\repl{k}ql} +u_{kq}R^{\repl{i}k}_{i\repl{k}pl} +u_{k}R^{\repl{i}k}_{i\repl{k}pl,q}+u_{ilpq}\\
&:=(\Xi_1)_{pqil}+u_{ilpq}.
\end{split}
\end{equation}
Then we get the result required by adding the above equations together.
\end{proof}
\begin{rem}\label{08230110}
In Lemma $6.3$ and Proposition $6.2$ of \cite{STW}, Smoczyk-Tsui-Wang used the function
\begin{equation}\label{08230121}
\widetilde\rho=\frac{1}{2}\log\frac{\det\eta_{ij}}{\det g_{ij}}
\end{equation}
to consider the $C^2$-estimate of $u$, which aims to prove that the smallness of $|D^2u|^2_g$ is kept along the generalized Lagrangian mean curvature flow. This makes sense because in their case $\hat\chi=0$, the smallness of $\widetilde\rho$ is equivalent to the smallness of $|D^2u|^2_g$. However, we can not work with $\widetilde\rho$ because there is no such equivalence between $\widetilde\rho$ and $|D^2u|^2_g$ due to $\hat \chi\neq 0$.
\end{rem}
Define
\begin{equation}\label{0823013}
\Theta= g^{ip}g^{jq}g^{kr}u_{ijk}u_{pqr}
\end{equation}
and
\begin{equation}\label{0823014}
\Upsilon= \eta^{ms}g^{ip}g^{jq}g^{kr}u_{ijkm}u_{pqrs}.
\end{equation}

\begin{lem}\label{0823015}
The evolution equation of $\Theta$ along the generalized mean curvature flow \eqref{GLMCF1} is
\begin{equation}\label{0823016}
\begin{split}
(\frac{\partial}{\partial t}-\Delta_\eta)\Theta&=-2\Upsilon+2g^{ip}g^{jq}g^{kr}(\eta^{ms}_{,jk}u_{msi} +\eta^{ms}_{,j}u_{msik}+\eta^{ms}_{,k}u_{msij})u_{pqr}\\
&\ \ \ +2g^{ip}g^{jq}g^{kr}(\eta^{ms}_{,jk}\hat \chi_{m,si} +\eta^{ms}_{,j}\hat \chi_{m,sik}+\eta^{ms}_{,k}\hat \chi_{m,sij}+\eta^{ms}\hat \chi_{m,sijk})u_{pqr}\\
&\ \ \ +2 \eta^{ms}g^{ip}g^{jq}g^{kr}(\Xi_2)_{msijk}u_{pqr},
\end{split}
\end{equation}
where
\begin{equation}\label{0823017}
\begin{split}
(\Xi_2)_{msijk} &=(u_{l}R^{\repl{m}l}_{m\repl{l}si})_{,jk}+(u_{lm}R^{\repl{i}l}_{i\repl{l}sj}+u_{il}R^{\repl{m}l}_{m\repl{l}sj})_{,k}+(u_{l}R^{\repl{i}l}_{i\repl{l}mj})_{,sk}\\
&\ \ \ +(u_{ljm}R^{\repl{i}l}_{i\repl{l}sk}+u_{ilm}R_{j\repl{l}sk}^{\repl{j}l} +u_{ijl}R^{\repl{m}l}_{m\repl{l}sk}) +(u_{lj}R^{\repl{i}l}_{i\repl{l}mk} +u_{il}R^{\repl{j}l}_{j\repl{l}mk})_{,s}\\
&=D u \ast_g D^2 Rm +D^2 u\ast_g D Rm+D^3 u \ast_g Rm.
\end{split}
\end{equation}
\end{lem}
\begin{proof}
Direct computations show
\begin{equation}\label{0823018}
\begin{split}
\frac{\partial }{\partial t}\Theta&=2g^{ip}g^{jq}g^{kr}\left(\frac{\partial }{\partial t}u\right)_{ijk}u_{pqr}=2g^{ip}g^{jq}g^{kr}(\theta(\hat \chi_u))_{ijk}u_{pqr}\\
&=2g^{ip}g^{jq}g^{kr}(\eta^{ms} \hat \chi_{m,s i})_{j k} u_{pqr} +2g^{ip}g^{jq}g^{kr}(\eta^{ms} u_{ms i})_{j k} u_{pqr}\\
&=2g^{ip}g^{jq}g^{kr}(\eta^{ms}_{,jk}\hat \chi_{m,si} +\eta^{ms}_{,j}\hat \chi_{m,sik}+\eta^{ms}_{,k}\hat \chi_{m,sij}+\eta^{ms}\hat \chi_{m,sijk})u_{pqr}\\
&\ \ \ +2g^{ip}g^{jq}g^{kr}(\eta^{ms}_{,jk}u_{msi} +\eta^{ms}_{,j}u_{msik}+\eta^{ms}_{,k}u_{msij}+\eta^{ms}u_{msijk})u_{pqr},
\end{split}
\end{equation}
and
\begin{equation}\label{0823019}
\begin{split}
\Delta_\eta \Theta&=\eta^{ms}g^{ip}g^{jq}g^{kr}(u_{ijk}u_{pqr})_{ms}\\
&=2\eta^{ms}g^{ip}g^{jq}g^{kr}(u_{ijkm}u_{pqrs}+u_{ijkms}u_{pqr}).
\end{split}
\end{equation}
Using the communication formulas, we obtain
\begin{equation}\label{0823020}
\begin{split}
&\ \ \ \ u_{msijk}-u_{ijkms}\\
&=(u_{msi}-u_{mis})_{jk}+(u_{imsj}-u_{imjs})_{k}+(u_{imj}-u_{ijm})_{sk}\\
&\ \ \ +(u_{ijmsk}-u_{ijmks})+(u_{ijmk}-u_{ijkm})_{s}\\
&=(u_{l}R^{\repl{m}l}_{m\repl{l}si})_{,jk}+(u_{lm}R^{\repl{i}l}_{i\repl{l}sj}+u_{il}R^{\repl{m}l}_{m\repl{l}sj})_{,k}+(u_{l}R^{\repl{i}l}_{i\repl{l}mj})_{,sk}\\
&\ \ \ +(u_{ljm}R^{\repl{i}l}_{i\repl{l}sk}+u_{ilm}R_{j\repl{l}sk}^{\repl{j}l} +u_{ijl}R^{\repl{m}l}_{m\repl{l}sk}) +(u_{lj}R^{\repl{i}l}_{i\repl{l}mk} +u_{il}R^{\repl{j}l}_{j\repl{l}mk})_{,s}\\
&:=(\Xi_2)_{msijk}.
\end{split}
\end{equation}
Then we get the result needed.
\end{proof}

\section{The long-time existence of the generalized Lagrangian mean curvature flow}
\label{sec-4}
In this section, we first prove that the smallness of $|D^2u|^2_g$ is preserved along the generalized mean curvature flow. Then by this property, we derive the $C^3$-estimate of $u$ and prove the long-time existence of flow \eqref{GLMCF1} when $|D^2u_0|^2_g$ is sufficiently small.
\subsection{Smallness of $|D^2u|^2_g$ along the generalized Lagrangian mean curvature flow}\label{sec-4.1}
If the graph $\Sigma_{\hat\chi}$ induced by $\hat \chi$ is a special Lagrangian submanifold in $T^*X$, that is, $\theta(\hat \chi)$ is a constant, we prove the following lemmas for $\tau=(u-u_0(p)-\theta(\hat \chi)t)^2$, $\vartheta$ and $\rho$.
\begin{lem}\label{0824001}
Along the generalized Lagrangian mean curvature flow \eqref{GLMCF1}, $\tau$ satisfies inequality
\begin{equation}\label{0824002}
(\frac{\partial }{\partial t}-\Delta_{\eta})\tau\leqslant C(1+\rho)\rho \tau^{\frac{1}{2}}+(-C_0+C\rho)\vartheta,
\end{equation}
where $C_0$ and $C$ are positive constants depending only on $n$, $g$ and $\hat\chi$.
\end{lem}
\begin{proof}
From Lemma \ref{0822053}, we have
\begin{equation}\label{0824003}
(\frac{\partial }{\partial t}-\Delta_{\eta})(u-u_0(p)-\theta(\hat \chi)t)^2=2(u-u_0(p)-\theta(\hat \chi)t)(\theta(\hat \chi_{u})-\theta(\hat \chi)-\Delta_{\eta}u)-2\eta^{pq}u_{p}u_{q}.
\end{equation}
We consider $\theta(\hat\chi_{su})$ as a smooth function on $s\in[0,1]$, where $\hat\chi_{su}=\hat\chi+sdu$. By using \eqref{0822041}, we have
\begin{equation}\label{0824004}
\begin{split}
\theta(\hat \chi_{u})-\theta(\hat \chi)-\Delta_{\eta}u&=\theta(\hat \chi_{1\cdot u})-\theta(\hat \chi_{0\cdot u})-\Delta_{\eta}u\\
&=\frac{d}{ds} \theta(\hat \chi_{su})\Big|_{s=\xi_t\in (0,1)}-\Delta_{\eta}u\\
&=\hat\eta^{pq}u_{pq} -\eta^{pq}u_{pq},
\end{split}
\end{equation}
where $\hat\eta^{pq}$ and $\eta^{pq}$ are the inverse of $\hat \eta_{pq}=g_{pq}+(\hat\chi_{\xi_t\cdot u})_{k,p}g^{kl}(\hat\chi_{\xi_t\cdot u})_{l,q}$ and $\eta_{pq}=g_{pq}+(\hat\chi_{u})_{k,p}g^{kl}(\hat\chi_{u})_{l,q}$ respectively. Moreover, 
\begin{equation}\label{08240041}
\hat\eta\geqslant g\ \ \ \text{and}\ \ \ \eta\geqslant g.
\end{equation}
Therefore, we have
\begin{equation}\label{0824005}
\begin{split}
&\ \ \ \theta(\hat \chi_{u})-\theta(\hat \chi)-\Delta_{\eta}u\\
&=\eta^{ps}(\eta_{sm}-\hat \eta_{sm})\hat\eta^{mq} u_{pq}\\
&=\eta^{ps}((1-\xi_t)(\hat\chi_{i,s}g^{ij}u_{jm}+u_{is}g^{ij}\hat\chi_{j,m}) +(1-\xi_t^2) u_{is}g^{ij}u_{jm})\hat\eta^{mq} u_{pq}\\
&\leqslant C(1+\rho)\rho.
\end{split}
\end{equation}
Then inequality \eqref{0824003} can be controlled as 
\begin{equation}\label{0824006}
(\frac{\partial }{\partial t}-\Delta_{\eta})\tau\leqslant C(1+\rho)\rho \tau^{\frac{1}{2}} -2\eta^{pq}u_{p}u_{q}.
\end{equation}
We deal with the term $-2\eta^{pq}u_{p}u_{q}$ as follows,
\begin{equation}\label{0824007}
\begin{split}
-2\eta^{pq}u_{p}u_{q}&=-2(\eta^{pq}-\hat \eta^{pq})u_{p}u_{q}-2\hat \eta^{pq }u_{p}u_{q}\\
&=2\eta^{pm}(\eta_{ms}-\hat \eta_{ms})\hat \eta^{sq}u_{p}u_{q} -2\hat \eta^{pq}u_{p}u_{q}\\
&=2\eta^{pm}(\hat\chi_{i,s}g^{ij}u_{jm}+u_{is}g^{ij}\hat\chi_{j,m} + u_{is}g^{ij}u_{jm})\hat \eta^{sq}u_{p}u_{q} -2\hat \eta^{pq}u_{p}u_{q}\\
&\leqslant-2C_0\vartheta+(\rho^{\frac{1}{2}}+\rho)\vartheta\\
&\leqslant(-C_0+C\rho)\vartheta,
\end{split}
\end{equation}
where $\hat\eta_{pq}=g_{pq}+(\hat\chi)_{k,p}g^{kl}(\hat\chi)_{l,q}$. We complete the proof.
\end{proof}

\begin{lem}\label{0824008}
Along the generalized Lagrangian mean curvature flow \eqref{GLMCF1}, there holds 
\begin{equation}\label{0824009}
(\frac{\partial }{\partial t}-\Delta_{\eta})\vartheta \leqslant (-C_1+C\rho)\rho+C\vartheta,
\end{equation}
where $C_1$ and $C$ are positive constants depending only on $n$,  $g$ and $\hat \chi$.
\end{lem}

\begin{proof}
From Lemma \ref{0823002}, the evolution equation of $\vartheta$ along the generalized mean curvature flow \eqref{GLMCF1} is given by
\begin{equation}\label{0824011}
\begin{split}
(\frac{\partial}{\partial t} -\Delta_\eta)\vartheta &=-2\eta^{pq}g^{ij} u_{ip}u_{jq}-2 g^{ij}\eta^{pq} R_{p\repl{l}iq}^{\repl{p}l} u_{l}u_{j} +2g^{ij}\eta^{pq} \hat \chi_{p,qi}u_{j}\\
&=-2\eta^{pq}g^{ij} u_{ip}u_{jq}-2 g^{ij}\eta^{pq} R_{p\repl{l}iq}^{\repl{p}l} u_{l}u_{j} +2g^{ij}(\eta^{pq}-\hat\eta^{pq})\hat \chi_{p,qi}u_{j},
\end{split}
\end{equation}
where we use the fact $\hat\eta^{pq}\hat\chi_{p,qi}=0$ (see \eqref{0822042}) in the second equality. For simplicity, we denote $\chi:=\hat\chi_u=\hat\chi+du$. Then we have
\begin{equation}\label{0824012}
\begin{split}
\eta^{pq}-\hat \eta^{pq}&=\eta^{pl}(\hat \eta_{sl}-\eta_{sl})\hat \eta^{sq}=\eta^{pl}(\hat \chi_{k,s}g^{kj}\hat \chi_{j,l} -\chi_{k,s}g^{kj}\chi_{j,l})\hat \eta^{sq}\\
&=-\eta^{pl}\hat \eta^{sq}(\hat \chi_{k,s}g^{kj}u_{jl}+ u_{ks}g^{kj}\hat \chi_{j,l}+u_{ks}g^{kj}u_{jl}),
\end{split}
\end{equation}
which implies that
\begin{equation}\label{0824013}
|\eta^{-1}-\hat \eta^{-1}|_g\leqslant C(\sqrt{\rho}+\rho ).
\end{equation}
The second term in the right hand side of \eqref{0824011} can be bouded as follows,
\begin{equation}\label{0824014}
-2 g^{ij}\eta^{pq} R_{p\repl{l}iq}^{\repl{p}l} u_{l}u_{j} \leqslant C\vartheta.
\end{equation}
By using the Cauchy-Schwarz inequality and \eqref{0824013}, for any $\varepsilon>0$, we have
\begin{equation}\label{0824015}
2g^{ij}(\eta^{pq}-\hat \eta^{pq})\hat \chi_{p,qi}u_{j}\leqslant \varepsilon(\rho+\rho^2)+C(\varepsilon)\vartheta.
\end{equation}
The same argument as in \eqref{0824007} imply that
\begin{equation}\label{0824017}
-2g^{ij}\eta^{pq}u_{ip}u_{jq}\leqslant (-C_0+C\rho)\rho.
\end{equation}
Taking $\varepsilon$ sufficiently small, we have
\begin{equation}\label{0824016}
\begin{split}
(\frac{\partial }{\partial t}-\Delta_{\eta})\vartheta &\leqslant (-C_0+C\rho)\rho+\varepsilon(1+\rho)\rho +C\vartheta\\
&\leqslant (-C_1+C\rho)\rho+C\vartheta,
\end{split}
\end{equation}
where $C_1$ and $C$ are positive constants depending only on $n$, $g$ and $\hat \chi$. We complete the proof.
\end{proof}
\begin{lem}\label{0824018}
The function $\rho$ satisfies the following inequality along the generalized Lagrangian mean curvature flow \eqref{GLMCF1},
\begin{equation}\label{0824019}
(\frac{\partial}{\partial t}-\Delta_\eta)\rho\leqslant (-1+C\rho)\eta^{pq}g^{ij}g^{kl}u_{ikq}u_{jlp}+C(1+\rho)\rho+C(1+\rho)\vartheta,
\end{equation}
where $C$ is a positive constant depending only on $n$, $g$ and $\hat \chi$.
\end{lem}

\begin{proof}
At a point $p\in X$, we choose the normal coordinate system with respect to $g$ such that
\begin{equation}\label{0824020}
g_{ij}(p)=\delta_{ij}\ \ \ \text{ and }\ \ \ u_{ij}(p)=\sigma_i\delta_{ij}.
\end{equation}
Then by Lemma \ref{0823006}, the evolution equation of $\rho$ along the generalized mean curvature flow \eqref{GLMCF1} is given by
\begin{equation}\label{0824021}
\begin{split}
(\frac{\partial }{\partial t}-\Delta_{\eta})\rho&=-2\eta^{pb}\eta^{aq}\eta_{ab,i}\hat \chi_{p,qi}\sigma_i +2\eta^{pq} \hat\chi_{p,qii}\sigma_i -2\eta^{pq} u_{ikp}u_{kiq} +2\sigma_i\eta^{pq}(\Xi_1)_{pq ii}\\
&\ \ \ -2\eta^{pb}\eta^{aq}\eta_{ab,i}u_{pqi}\sigma_i.
\end{split}
\end{equation}
We deal with the first two terms as follows,
\begin{equation}\label{0824022}
\begin{split}
&\ \ \ -2\eta^{pb}\eta^{aq}\eta_{ab,i}\hat \chi_{p,qi}\sigma_i +2\eta^{pq} \hat\chi_{p,qii}\sigma_i\\
&=-2\eta^{pb}\eta^{aq}(\eta_{ab,i}-\hat \eta_{ab,i})\hat \chi_{p,qi}\sigma_i-2\eta^{pb}(\eta^{aq}-\hat\eta^{aq})\hat\eta_{ab,i}\hat \chi_{p,qi}\sigma_i\\
&\ \ \ -2(\eta^{pb}-\hat \eta^{pb})\hat \eta^{aq}\hat \eta_{ab,i} \hat \chi_{p,qi}\sigma_i+2(\eta^{pq}-\hat \eta^{pq})\hat\chi_{p,qii}\sigma_i\\
&\ \ \ -2\hat \eta^{pb}\hat \eta^{aq}\hat \eta_{ab,i} \hat \chi_{p,qi}\sigma_i+2\hat \eta^{pq}\hat\chi_{p,qii}\sigma_i.
\end{split}
\end{equation}
Since $\hat\theta=\theta(\hat\chi)$ is a constant, according to \eqref{0822042}, we have
\begin{equation}\label{08240221}
-2\hat \eta^{pb}\hat \eta^{aq}\hat \eta_{ab,i} \hat \chi_{p,qi}\sigma_i+2\hat \eta^{pq}\hat\chi_{p,qii}\sigma_i=2(\hat\eta^{pq}\hat\chi_{p,qi})_i\sigma_i=0.
\end{equation}
Same arguments as in \eqref{0824012} imply that
\begin{equation}\label{0824023}
\begin{split}
\eta^{pq}-\hat \eta^{pq}&=\eta^{pl}(\hat \eta_{sl}-\eta_{sl})\hat \eta^{sq}\\
&=-\eta^{pl}\hat \eta^{sq}(\hat \chi_{k,s}u_{kl}+ u_{ks}\hat \chi_{k,l}+u_{ks}u_{kl})
\end{split}
\end{equation}
and that
\begin{equation}\label{0824024}
\begin{split}
(\eta_{ab,i}-\hat \eta_{ab,i})&=\hat \chi_{k,ai}u_{kb} +u_{kai}\hat\chi_{k,b} +u_{kai}u_{kb}+\hat \chi_{k,a}u_{kbi} +u_{ka}\hat\chi_{k,bi} +u_{ka}u_{kbi}\\
&=\sigma_a(\hat \chi_{a,bi}+u_{abi})+\sigma_b(\hat \chi_{b,ai}+u_{bai})+u_{kai}\hat\chi_{k,b}+\hat \chi_{k,a}u_{kbi}\\
&=\sigma_a(\hat \chi_{a,bi}+u_{aib})+\sigma_b(\hat \chi_{b,ai}+u_{bia})+\sigma_au_lR^{\repl{a}l}_{a\repl{l}bi}+\sigma_bu_lR^{\repl{b}l}_{b\repl{l}ai}\\
&\ \ \ +u_{kia}\hat\chi_{k,b}+u_{l}R_{k\repl{l}ai}^{\repl{k}l}\hat\chi_{k,b} +\hat \chi_{k,a}u_{k i b}+u_{l}R^{\repl{k}l}_{k\repl{l}bi}\hat \chi_{k,a}.
\end{split}
\end{equation}
By using the Cauchy-Schwarz inequality and \eqref{0824022}-\eqref{0824024}, the first two terms on the right hand side of \eqref{0824021} can be controlled as 
\begin{equation}\label{08240241}
-2\eta^{pb}\eta^{aq}\eta_{ab,i}\hat \chi_{p,qi}\sigma_i +2\eta^{pq} \hat\chi_{p,qii}\sigma_i\leqslant \frac{1}{2}\eta^{pq}u_{ijp}u_{ijq} +C(\rho+\rho^2)+C\vartheta.
\end{equation}
Since $\Xi_1=D^2u\ast_g Rm+Du \ast_g DRm$, the forth term on the right hand side of \eqref{0824021} can be controlled as
\begin{equation}\label{0824025}
2\sigma_i\eta^{pq}(\Xi_1)_{pqii}\leqslant C\rho+C\vartheta.
\end{equation}
Then we deal with the last term in \eqref{0824021}. Indeed,
\begin{equation}\label{0824026}
\begin{split}
&\ \ \ \ -2\eta^{pb}\eta^{aq}\eta_{ab,i}u_{pqi}\sigma_i\\
&=-2\eta^{pb}\eta^{aq}\eta_{ab,i}\sigma_i (u_{piq}+u_{l}R^{\repl{p}l}_{p\repl{l}qi})\\
&=-2\eta^{pb}\eta^{aq}\sigma_i (u_{piq}+u_{l}R^{\repl{p}l}_{p\repl{l}qi})(\hat\chi_{s,ai}\chi_{s,b}+\chi_{s,a}\hat \chi_{s,bi})\\
&\ \ \ -2\eta^{pb}\eta^{aq}\sigma_i (u_{piq}+u_{l}R^{\repl{p}l}_{p\repl{l}qi})(u_{sai}\chi_{s,b}+\chi_{s,a}u_{sbi})\\
&=-2\eta^{pb}\eta^{aq}\sigma_i (u_{piq}+u_{l}R^{\repl{p}l}_{p\repl{l}qi})(\hat\chi_{s,ai}\chi_{s,b}+\chi_{s,a}\hat \chi_{s,bi})\\
&\ \ \ -2\eta^{pb}\eta^{aq}\sigma_i (u_{piq}+u_{l}R^{\repl{p}l}_{p\repl{l}qi})(u_{si a}\chi_{s,b}+u_{l}R^{\ l}_{s\ ai}\chi_{s,b}+\chi_{s,a}u_{sib}+u_{l}R^{\ l}_{s\ bi}\chi_{s,a})\\
&\leqslant (\frac{1}{2}+C\rho)\eta^{pq}u_{ijq}u_{ijp} +C(1+\rho)\rho+C(1+\rho)\vartheta.
\end{split}
\end{equation}
Therefore, we have the following inequality for $\rho$,
\begin{equation}\label{0824027}
(\frac{\partial}{\partial t}-\Delta_\eta)\rho\leqslant (-1+C\rho)\eta^{pq}u_{ijq}u_{ijp}+C(1+\rho)\rho+C(1+\rho)\vartheta,
\end{equation}
where $C$ is a positive constant depending only on $n$, $g$ and $\hat \chi$.
\end{proof}

Then we prove that the smallness of $D^2u$ is preserved along the generalized Lagrangian mean curvature flow \eqref{GLMCF1}. We define the following auxiliary function $Q$,
\begin{equation}\label{0824028}
Q=\rho +K_1\vartheta+K_2 \tau,
\end{equation}
where $K_1$ and $K_2$ are positive constants to be determined later. Since $X$ is a compact Riemannian manifold, we have the following lemma according to differential mean value theorem.

\begin{lem}\label{0824030}
There exists a positive constant $C_g$ depending only on $n$ and $g$ such that
\begin{equation}\label{0824029}
\begin{cases}
Q\leqslant C_g\max\limits_X\rho,&\text{ at }t=0,\\
\rho\leqslant Q, &\text{ if }t\geqslant 0.
\end{cases}
\end{equation}
\end{lem}
According to Lemma \ref{0824030}, we only need to prove the smallness of $Q$ is kept along the generalized Lagrangian mean curvature flow \eqref{GLMCF1} instead of $\rho$.
\begin{thm}\label{0824031}
There exists a constant $\delta_0>0$ such that if $\rho\leqslant \delta_0$ at $t=0$, then $\rho\leqslant 2C_g\delta_0$ along the generalized Lagrangian mean curvature flow \eqref{GLMCF1}, where $C_g$ is the constant in Lemma \ref{0824030}.
\end{thm}
\begin{proof}
By Lemma \ref{0824001}, \ref{0824008} and \ref{0824018}, we have the following inequality for $Q$,
\begin{equation}\label{0824032}
\begin{split}
(\frac{\partial}{\partial t}-\Delta_\eta)Q\leqslant &(-1+C\rho)\eta^{pq}u_{ijq}u_{ijp} +C(1+\rho)\rho +C(1+\rho)\vartheta\\
&+K_1(-C_1+C\rho)\rho+K_1C\vartheta +K_2C(1+\rho)\rho\tau^{\frac{1}{2}}+K_2(-C_0+\rho)\vartheta\\
=& (-1+C\rho)\eta^{pq}u_{ijq}u_{ijp}+(-K_2C_0+K_1C+C+K_2\rho+C\rho)\vartheta\\
&+(-K_1C_1+C+C\rho+K_1C\rho+K_2C\tau^{\frac{1}{2}}+K_2C\rho \tau^{\frac{1}{2}})\rho,
\end{split}
\end{equation}
where $C_0$, $C_1$ and $C$ are positive constants depending only on $n$, $g$ and $\hat \chi$. We choose $K_1$ and $K_2$ satisfying
\begin{equation}\label{0824033}
-K_1C_1+C=-2\ \ \ \text{ and }\ \ \ -K_2C_0+K_1C+C=-1.
\end{equation}
Then
\begin{equation}\label{0824034}
(\frac{\partial}{\partial t}-\Delta_\eta)Q\leqslant(-1+A_1Q)\eta^{pq}u_{ijq}u_{ijp} +(-1+A_2Q)\vartheta +(-1+A_3Q^{\frac{3}{2}})\rho.
\end{equation}
Let $\delta'>0$ be a constant such that 
\begin{equation}\label{0824035}
-1+A_1\delta'<0,\ \ \ -1+A_2\delta'<0\ \ \ \text{ and }\ \ \ -1+A_3\delta'^{\frac{3}{2}}<0.
\end{equation}
We claim that if $\rho\leqslant\frac{\delta'}{2C_g}$ at $t=0$, then $Q<\delta'$ for all time $t\in[0,T)$ where the flow \eqref{GLMCF1} exists. 

First by Lemma \ref{0824030}, we have $Q\leqslant\frac{\delta'}{2}$ at $t=0$. If there exists a time $T_0<T$ such that
\begin{equation}\label{0824036}
Q(t)<\delta'\ \ \ on\ \ \ [0,T_0)\ \ \ \text{ and }\ \ \ Q(T_0)=\delta'.
\end{equation}
By \eqref{0824034}, \eqref{0824035} and the maximum principle, since $Q\leqslant\frac{\delta'}{2}$ at $t=0$, $Q\leqslant \frac{\delta'}{2}$ on $[0,T_0)$ and then $Q(T_0)\leqslant\frac{\delta'}{2}$, which is a contradiction with \eqref{0824036}. Hence $\rho\leqslant Q<\delta'$ for all time $t\in[0,T)$. Taking $\delta_0=\frac{\delta'}{2C_g}$, we complete the proof.
\end{proof}

\subsection{Third-order estimate}
Assume that $Q$ is uniformly bounded on $[0,T)$, that is, there is a uniform positive constant $L$ such that $Q\leqslant L$ for $t\in[0,T)$. We first derive the following estimates for $\Theta$. 
\begin{lem}\label{0824037}
If $Q$ is uniformly bounded by a positive constant $L$ for all $t\in[0,T)$, then $\Theta$ satisfies the following inequality
\begin{equation}\label{0824038}
(\frac{\partial}{\partial t}-\Delta_\eta)\Theta\leqslant -\Upsilon+C\Theta^2+C,
\end{equation}
where $C$ is a positive constant depending only on $n$, $\hat \chi$, $g$ and $L$.
\end{lem}
\begin{proof}
From Lemma \ref{0823015}, $\Theta$ satisfies the following evolution equation
\begin{equation}\label{0824039}
\begin{split}
(\frac{\partial}{\partial t}-\Delta_\eta)\Theta&=-2\Upsilon+2g^{ip}g^{jq}g^{kr}(\eta^{ms}_{,jk}u_{msi} +\eta^{ms}_{,j}u_{msik}+\eta^{ms}_{,k}u_{msij})u_{pqr}\\
&\ \ \ +2g^{ip}g^{jq}g^{kr}(\eta^{ms}_{,jk}\hat \chi_{m,si} +\eta^{ms}_{,j}\hat \chi_{m,sik}+\eta^{ms}_{,k}\hat \chi_{m,sij}+\eta^{ms}\hat \chi_{m,sijk})u_{pqr}\\
&\ \ \ +2 \eta^{ms}g^{ip}g^{jq}g^{kr}(\Xi_2)_{msijk}u_{pqr},
\end{split}
\end{equation}
where
\begin{equation}\label{0824040}
\Xi_2 =Du \ast_g D^2 Rm +D^2 u\ast_g D Rm+D^3 u \ast_g Rm.
\end{equation}
The uniform bound of $Q$ implies that $D^2 u$ is uniformly bounded as an application of Lemma \ref{0824030}. Hence $D\chi$ is uniformly bounded and then there exists a constant $C>1$ depending only on $n$, $\hat \chi$, $g$ and $L$, such that
\begin{equation}\label{0906001}
g\leqslant \eta\leqslant Cg.
\end{equation}
Since $\chi$ is closed, $\chi_{i,j}=\chi_{j,i}$. For convenience, we choose normal coordinate system near $p\in X$ such that
\begin{equation}\label{0824041}
g_{ij}(p)=\delta_{ij}\ \ \ \text{ and }\ \ \ \chi_{i,j}=\delta_{ij}\mu_i.
\end{equation}
Since $\eta_{ij}=g_{ij}+\chi_{p,i}g^{pq}\chi_{q,j}$, $\eta_{ij}=\delta_{ij}\nu_i$, where $\nu_i=1+\mu_i^2$. Then we have the following expressions of $\eta^{ms}_{,j}$ and $\eta^{ms}_{,jk}$ at $p$,
\begin{equation}\label{0824042}
\begin{split}
\eta^{ms}_{,j}=&-\eta^{mb}\eta^{as}(\chi_{p,aj}\chi_{p,b}+\chi_{p,a}\chi_{p,bj})\\
=&-2\nu^m\nu^s\mu_m\chi_{m,sj}
\end{split}
\end{equation}
and
\begin{equation}\label{0824043}
\begin{split}
\eta^{ms}_{,jk}&=-\eta^{mb}\eta^{as}(\chi_{p,ajk}\chi_{p,b}+\chi_{p,a}\chi_{p,bjk}+\chi_{p,aj}\chi_{p,bk} +\chi_{p,ak}\chi_{p,bj})\\
&\ \ \ +\eta^{mc}\eta^{db}\eta^{as}(\chi_{r,ck}\chi_{r,d}+\chi_{r,c}\chi_{r,dk}) (\chi_{p,aj}\chi_{p,b}+\chi_{p,a}\chi_{p,bj})\\
&\ \ \ +\eta^{mb}\eta^{ac}\eta^{ds}(\chi_{r,ck}\chi_{r,d}+\chi_{r,c}\chi_{r,dk})(\chi_{p,aj}\chi_{p,b}+\chi_{p,a}\chi_{p,bj})\\
&=-2\nu^m\nu^s\mu_m\chi_{m,sjk}-2\nu^m\nu^s\chi_{p,sj}\chi_{p,mk}\\
&\ \ \ +4\nu^m\nu^d\nu^s\mu_m\mu_s\chi_{m,dk}\chi_{s,dj}+4\nu^m\nu^c\nu^s\mu_s\mu_m\chi_{s,ck}\chi_{m,cj}.
\end{split}
\end{equation}
By using the Cauchy-Schwarz inequality, we have
\begin{equation}\label{0824044}
2(\eta^{ms}_{,jk}u_{msi} +\eta^{ms}_{,j}u_{msik}+\eta^{ms}_{,k}u_{msij})u_{ijk}\leqslant \frac{1}{2}\Upsilon +C\Theta^2+C,
\end{equation}
\begin{equation}\label{0824045}
2(\eta^{ms}_{,jk}\hat \chi_{m,si} +\eta^{ms}_{,j}\hat \chi_{m,sik}+\eta^{ms}_{,k}\hat \chi_{m,sij}+\eta^{ms}\hat \chi_{m,sijk})u_{ijk}\leqslant \frac{1}{2}\Upsilon +C\Theta^2+C
\end{equation}
and
\begin{equation}\label{0824046}
\begin{split}
\eta^{ms}(\Xi_2)_{msijk}u_{ijk}=&\eta^{-1}\ast_g D^3 u\ast_g(Du \ast_g D^2 Rm +D^2 u\ast_g D Rm+D^3 u \ast_g Rm)\\
\leqslant &C+C\Theta^2,
\end{split}
\end{equation}
where $C$ is a positive constant depending only on $n$, $\hat \chi$, $g$ and $L$. Putting \eqref{0824044}-\eqref{0824046} into equation \eqref{0824039}, we complete the proof. 
\end{proof}
To prove the uniformly higher order estimates for $u$, we need the smallness of $D^2 u$ along the generalized Lagrangianmean curvature flow. We first prove the following lemma if $D^2u$ is sufficiently small.
\begin{lem}\label{0824047}
There exists a constant $\delta_0>0$ such that if $\rho\leqslant \delta_0$ at time $t=0$, then there holds
\begin{equation}\label{0824048}
(\frac{\partial}{\partial t}-\Delta_\eta)\rho \leqslant -\frac{1}{4}\Theta+C.
\end{equation}
\end{lem}
\begin{proof}
We only need to take $\delta_0$ in Theorem \ref{0824031} sufficiently small such that $\eta$ sufficiently close to $g$ and that in Lemma \ref{0824018}, we also have
\begin{equation}\label{0824049}
(-1+C\rho)\eta^{pq}g^{ij}g^{kl}u_{ikq}u_{jlp}\leqslant -\frac{1}{4}g^{pq}g^{ij}g^{kl}u_{ikq}u_{jlp},
\end{equation}
where $C$ and $C_g$ are the constants appeared in Lemma \ref{0824018} and Theorem \ref{0824031}. 
\end{proof}
As an application of Lemma \ref{0824037} and Lemma \ref{0824047}, we obtain the uniform estimate of $\Theta$.
\begin{thm}\label{0824050}
There exists a constant $\delta_0>0$ such that if $\rho\leqslant \delta_0$ at $t=0$, then $\Theta$ is uniformly bounded along the generalized  Lagrangian mean curvature flow \eqref{GLMCF1}.
\end{thm}
\begin{proof}
To get the uniform estimate of $\Theta$, we consider the following auxiliary function
\begin{equation}\label{0824051}
\Gamma=e^{A\rho}\Theta,
\end{equation}
where $A$ is a constant to be determined later. We first choose $\delta_0$ to be the constant in Lemma \ref{0824047}. Then $\rho$ is bounded by $2C_g\delta_0$ by Theorem \ref{0824031} and satisfies
\begin{equation}\label{0824052}
(\frac{\partial }{\partial t}-\Delta_\eta)\rho\leqslant -\frac{1}{4}\Theta +C.
\end{equation}
From Lemma \ref{0824037}, we have
\begin{equation}\label{0824053}
(\frac{\partial}{\partial t}-\Delta_\eta) \Theta\leqslant -\Upsilon+C\Theta^2 +C.
\end{equation}
Combing the inequalities \eqref{0824052} and \eqref{0824053}, we obtain
\begin{equation}\label{0824054}
\begin{split}
(\frac{\partial}{\partial t}-\Delta_\eta) \Gamma&=Ae^{A\rho}\Theta(\frac{\partial}{\partial t}-\Delta_\eta)\rho +e^{A\rho}(\frac{\partial}{\partial t}-\Delta_\eta) \Theta-2Ae^{A\rho}\eta^{ij}\rho_i\Theta_j-A^2e^{A\rho}\Theta \eta^{ij}\rho_i\rho_j\\
&\leqslant A\Gamma (-\frac{1}{4}\Theta +C) +e^{A\rho}(-\Upsilon+C \Theta^2+C)-2Ae^{A\rho}\eta^{ij}\rho_i\Theta_j-A^2\Gamma \eta^{ij}\rho_i\rho_j\\
&=-2A\eta^{ij}\rho_i\Gamma_j+A\Gamma (-\frac{1}{4} \Theta +C) +e^{A\rho}(-\Upsilon+C\Theta^2+C)+A^2\Gamma \eta^{ij}\rho_i\rho_j,
\end{split}
\end{equation}
where we have used the following equality in the last equality
\begin{equation}\label{0824055}
D\Gamma=A\Gamma D\rho+e^{A\rho}D\Theta.
\end{equation}
Since
\begin{equation}\label{0824056}
\eta^{ij}\rho_i\rho_j=4\eta^{ij}u_{pqi}u_{pq}u_{slj}u_{sl}\leqslant 5\rho \Theta,
\end{equation}
we have
\begin{equation}\label{0824057}
\begin{split}
(\frac{\partial}{\partial t}-\Delta_\eta) \Gamma&\leqslant-2A\eta^{ij}\rho_i\Gamma_j+A\Gamma (-\frac{1}{4} \Theta +C) +e^{A\rho}(-\Upsilon+C\Theta^2+C)+5A^2\rho e^{A\rho}\Theta^2\\
&=-2A\eta^{ij}\rho_i\Gamma_j+e^{A\rho}(\Theta^2(5\rho A^2-\frac{1}{4}A+C) + AC\Theta+C)-e^{A\rho}\Upsilon\\
&\leqslant-2A\eta^{ij}\rho_i\Gamma_j+e^{A\rho}(\Theta^2(10A^2C_g\delta_0-\frac{1}{4}A+C) + AC\Theta+C).
\end{split}
\end{equation}
We choose $\delta_0$ much smaller such that
\begin{equation}\label{0824059}
\frac{1}{16}-40C_g\delta_0C>0.
\end{equation}
Then there exists a positive constant $A$ such that
\begin{equation}\label{0907001}
-C_2:=10A^2C_g\delta_0-\frac{1}{4}A+C<0.
\end{equation}
As a consequence, we have the following inequality for $\Gamma$,
\begin{equation}\label{0824060}
\begin{split}
(\frac{\partial}{\partial t}-\Delta_\eta) \Gamma
\leqslant&-2A\eta^{ij}D_i\rho D_j \Gamma +e^{A\rho}(-C_2\Theta^2+C\Theta+C).
\end{split}
\end{equation}
By using the maximum principle and the uniform bound on $\rho$, we conclude that $\Theta$ is uniformly bounded along the generalized  Lagrangian mean curvature flow \eqref{GLMCF1}.
\end{proof}

\subsection{The long-time existence}
From Theorem \ref{0824050}, it is clearly that $\eta$ is $C^1$-bounded uniformly. Then the standard parabolic Schauder estimates give us all uniform higher order estimates of $u$. Hence we can get the long-time existence of the generalized Lagrangian mean curvature flow.

\begin{thm}\label{0824061}
There exists a constant $\delta_0>0$ such that if $\rho\leqslant \delta_0$ at $t=0$, then the generalized Lagrangian mean curvature flow \eqref{GLMCF1} exists for all $t\in[0,+\infty)$.
\end{thm}

\begin{proof}
We assume that the maximal existence time of the generalized Lagrangian mean curvature flow \eqref{GLMCF1} is $[0,T)$. According to Theorem \ref{0824050}, we know that $\Theta$ is uniformly bounded in $[0,T)$ and thus $\eta$ is uniformly $C^1$-bounded. Then the standard parabolic Schauder estimates imply all higher order uniform estimates of $u$. Hence we can extend the flow across time $T$ by the short-time existence if $T<+\infty$, which implies that $T$ must be $+\infty$.
\end{proof}

\section{Exponential Convergence}\label{sec-5}
In this section, we prove that the $1$-form $\hat \chi_{u_t}$ converges exponentially to $\hat \chi$ along the generalized Lagrangian mean curvature flow $(\ref{GLMCF1})$. First, along the generalized Lagrangian mean curvature flow, the function $\dot u:= \frac{\partial}{\partial t}u$ evolves as 
\begin{equation}\label{0815001}
\frac{\partial }{\partial t}\dot u =\eta^{i j} \dot u_{ij}.
\end{equation}
According to the maximum principle, $\dot u$ is bounded by $\|\dot u(0)\|_{C^0(X)}$.

\subsection{Harnack-type inequality}
In this subsection, we prove a Harnack-type inequality for the positive solution to the following parabolic equation
\begin{equation}\label{0815002}
 \frac{\partial v}{\partial t}=\eta^{ij}v_{ij},
\end{equation}
where $\{\eta^{ij}\}$ is the inverse matrix of $\{\eta_{ij}\}$, $\eta_{ij}=g_{ij}+(\hat\chi_{u})_{k,i}g^{kl}(\hat\chi_{u})_{l,j}$, $\hat\chi_{u}=\hat\chi+du$ and $u$ is the solution to equation $(\ref{GLMCF1})$. Set $f=\log v$ and
\begin{equation}\label{0815003}
F=t(\eta^{ij}f_i f_{j} -\alpha \dot f),
\end{equation}
where $\alpha\in(1,2)$ is a constant. By Equation \eqref{0815002}, we have
\begin{equation}\label{0815004}
\dot f -\eta^{ij}f_{ij}=\eta^{ij}f_i f_{j}
\end{equation}
and
\begin{equation}\label{0815005}
F=-t\eta^{ij} f_{ij} -t (\alpha-1)\dot f.
\end{equation}

\begin{lem}
\label{08150006}
There exists a positive constant $C$ depending only on $g$, $\hat\chi$ and $\|u\|_{C^4(X)}$ such that $F$ satisfies the following inequality
\begin{equation}\label{0815007}
  \eta^{kl} F_{kl} -\dot F\geqslant \frac{t}{2n}(\eta^{ij}f_i f_{j} -\dot f)^2 -2\eta^{ij} f_iF_j -(\eta^{ij}f_i f_{j} -\alpha \dot f) -C t \eta^{ij} f_i f_{j} -C t.
\end{equation}
\end{lem}
\begin{proof}
Direct computations show that
\begin{equation}\label{0815008}
\dot F=\eta^{ij} f_if_{j} -\alpha \dot f+2 t\eta^{ij}f_{j} \dot f_i+t \frac{\partial \eta^{ij}}{\partial t} f_i f_{j} -\alpha t\ddot f
\end{equation}
and
\begin{equation} \label{0815009}
\eta^{kl}F_{kl}=t\eta^{kl}\Big(2\eta^{ij} f_{ik}f_{jl} +4\eta^{ij}_{,k}f_{il}f_{j} +2\eta^{ij}f_{ikl}f_{j} +\eta^{ij}_{,k l}f_if_{ j} -\alpha \dot f_{kl}\Big).
\end{equation}
Since $u$ is uniformly bounded along the flow $(\ref{GLMCF1})$, we directly have
\begin{equation} \label{0815013}
t\big|\eta^{kl}\eta^{ij}_{,k l}f_if_{ j}\big|\leqslant Ct\eta^{i j}f_if_{ j},
\end{equation}
where $C$ is a positive constant depending only on $n$, $g$, $\hat\chi$ and $\|u\|_{C^4(X)}$. 

By the Cauchy-Schwarz inequality, we have
\begin{equation} \label{0815010}
|4t\eta^{kl}\eta^{ij}_{,k}f_{il}f_{j}|\leqslant \frac{2C t}{\varepsilon}\eta^{i j}f_if_{ j}+2\varepsilon t \eta^{kl}\eta^{ij} f_{ik}f_{jl},
\end{equation}
where $C$ is a positive constant depending only on $n$, $g$, $\hat\chi$ and $\|u\|_{C^3(X)}$. 

By the Cauchy-Schwarz inequality, \eqref{0815005} and \eqref{0815008}, we have
\begin{equation} \label{0815011}
\begin{split}
2t\eta^{kl}\eta^{ij}f_{ikl}f_{j}&=2t\eta^{ij}\eta^{kl}(f_{kli}f_{j} +f_a R_{k\ il}^{\ a}f_{j})\\
&\geqslant-Ct\eta^{i j}f_if_{ j} +2 t\eta^{i j}f_{ j}(\eta^{k l}f_{kl})_{i} -2t\eta^{ij}f_j\eta^{kl}_{,i}f_{kl}\\
&\geqslant -Ct\eta^{i j}f_if_{j} +2 t\eta^{i j}f_{ j}(\eta^{k l}f_{kl})_{i}-\frac{Ct}{\varepsilon}\eta^{i j}f_if_{j} -\varepsilon t \eta^{kl}\eta^{ij} f_{ik}f_{jl}\\
&=-Ct\eta^{i j}f_if_{j} -2\eta^{ij}f_{j}F_{i} -2t(\alpha-1)\eta^{ij}f_{j}\dot f_i -\frac{Ct}{\varepsilon}\eta^{i j}f_if_{j} -\varepsilon t \eta^{kl}\eta^{ij} f_{ik}f_{jl}\\
&=-Ct\eta^{i j}f_if_{j} -2\eta^{ij}f_{j}F_{i} -(\alpha-1)\dot F +(\alpha-1) (\eta^{ij}f_if_{ j} -\alpha \dot f)\\
&\ \ \ +(\alpha-1) t \frac{\partial \eta^{ij}}{\partial t} f_i f_{j} -\alpha(\alpha-1)t\ddot f-\frac{Ct}{\varepsilon}\eta^{i j}f_if_{j} -\varepsilon t \eta^{kl}\eta^{ij} f_{ik}f_{jl}\\
&\geqslant -Ct\eta^{i j}f_if_{j} -2\eta^{ij}f_{j}F_{i} -(\alpha-1)\dot F +(\alpha-1) (\eta^{ij}f_if_{ j} -\alpha \dot f)\\
&\ \ \ -\alpha(\alpha-1)t\ddot f-\frac{Ct}{\varepsilon}\eta^{i j}f_if_{j} -\varepsilon t \eta^{kl}\eta^{ij} f_{ik}f_{jl},
\end{split}
\end{equation}
and
\begin{equation} \label{0815012}
\begin{split}
-\alpha t\eta^{kl}\dot f_{kl}&=-\alpha t(\frac{F}{t^2}-\frac{\dot F}{t}-(\alpha-1)\ddot f)+\alpha tf_{kl}\frac{\partial\eta^{kl}}{\partial t}\\
&\geqslant -\frac{Ct}{\varepsilon}-\varepsilon t \eta^{kl}\eta^{ij} f_{ik}f_{jl} -\frac{\alpha F}{t}+\alpha \dot F +t\alpha(\alpha-1)\ddot f,
\end{split}
\end{equation}
where $C$ is a positive constant depending only on $n$, $g$, $\hat\chi$ and $\|u\|_{C^4(X)}$. 

Combining \eqref{0815009}-\eqref{0815012} together, we get that
\begin{equation}\label{0815014}
\begin{split}
\eta^{ij} F_{ij}&\geqslant\dot F-2\eta^{ij}F_i f_{j} -(\eta^{ij}f_i f_{j}- \alpha \dot f) +2t(1-2\varepsilon) \eta^{ij}\eta^{kl} f_{il}f_{kj}\\
&\ \ \ -Ct(1+\frac{1}{\varepsilon})\eta^{ij} f_i f_{ j} -\frac{Ct}{\varepsilon}.
\end{split}
\end{equation}
Taking $\varepsilon\leqslant\frac{3}{8}$ and applying the arithmetic-geometric mean inequality and equality \eqref{0815004}
\begin{equation}\label{0815015}
\eta^{ij} \eta^{kl} f_{i l}f_{k j}\geqslant \frac{1}{n} \left(\eta^{ij}f_{ij}\right)^2 =\frac{1}{n}(\dot f-\eta^{ij} f_i f_{j})^2,
\end{equation}
we obtain that
\begin{equation} \label{0815016}
\begin{split}
\eta^{ij} F_{ij}-\dot F&\geqslant\frac{t}{2n} (\dot f-\eta^{ij} f_i f_{j})^2 -2\eta^{ij}F_i f_{j} -(\eta^{ij}f_i f_{j}- \alpha \dot f) \\
&\ \ \ -Ct\eta^{i j} f_i f_{ j} -Ct,
\end{split}
\end{equation}
where $C$ is a positive constant depending only on $n$, $g$, $\hat\chi$ and $\|u\|_{C^4(X)}$. We complete the proof.
\end{proof}

\begin{lem}
\label{0815017}
There exists a positive constant $C$ depending only on $\alpha$, $n$, $g$, $\hat\chi$ and $\|u\|_{C^4(X)}$ such that for any $t>0$,
\begin{equation} \label{0815018}
\eta^{ij}f_if_{j} -\alpha \dot f\leqslant C+\frac{C}{t}.
\end{equation}
\end{lem}

\begin{proof}
For any $T>0$, let $(x_0,t_0)$ be the maximum point of $F$ on $X\times [0,T]$. If $t_0=0$, \eqref{0815018} can be deduced directly. We only need to consider the case $t_0>0$. At $(x_0,t_0)$, by Lemma \ref{08150006},
\begin{equation}
\label{0815019}
\frac{t_0}{2n}(\eta^{ij}f_i f_{j} -\dot f)^2-(\eta^{ij}f_i f_{j}-\alpha \dot f) \leqslant C t_0 \eta^{ij} f_i f_{j} +Ct_0.
\end{equation}

If $\dot f(x_0,t_0)>0$, since $\alpha\in(1,2)$, we have 
\begin{equation} \label{0815020}
\frac{t_0}{2n}(\eta^{ij}f_i f_{j} -\dot f)^2-(\eta^{ij}f_i f_{j}-\dot f) \leqslant C t_0 \eta^{ij} f_i f_{j} +Ct_0,
\end{equation}
which implies that at $(x_0,t_0)$, by using the Cauchy-Schwarz inequality, we have
\begin{equation} \label{0815021}
\eta^{ij} f_i f_{j} -\dot f\leqslant C \sqrt{\eta^{i j} f_i f_{ j}} +C +\frac{C}{t_0} \leqslant \left(1-\frac{1}{\alpha}\right) \eta^{ij} f_i f_{ j} +C+\frac{C}{t_0},
\end{equation}
that is,
\begin{equation} \label{0815022}
\eta^{ij}f_i f_{j}-\alpha \dot f\leqslant C+\frac{C}{t_0},
\end{equation}
where $C$ is a positive constant depending only on $\alpha$, $n$, $g$, $\hat\chi$ and $\|u\|_{C^4(X)}$. Then for any $x\in X$, 
\begin{equation} \label{0815023}
F(x,T)\leqslant F(x_0,t_0)= t_0(\eta^{i j}f_i f_{ j}-\alpha \dot f)\leqslant Ct_0 +C\leqslant CT+C.
\end{equation}
Therefore, we have
\begin{equation} \label{0815024}
(\eta^{ij}f_i f_{j} -\alpha \dot f)(x,T) \leqslant C+\frac{C}{T}\ \ \ on\ X.
\end{equation}
Since $T$ is arbitrary, we get \eqref{0815018}.

If $\dot f(x_0,t_0)\leqslant 0$. By \eqref{0815019}, at $(x_0,t_0)$, we have
\begin{equation}
\label{0815025}
\frac{t_0}{2n}(\eta^{i j} f_i f_{j})^2 -\eta^{ij} f_i f_{j} +\alpha \dot f \leqslant Ct_0 \eta^{ij} f_i f_{j} +Ct_0,
\end{equation}
that is,
\begin{equation}
\label{0815026}
\frac{1}{2n}(\eta^{ij}f_i f_{j})^2 -(\frac{1}{t_0} +C)\eta^{ij}f_i f_{j}\leqslant C -\frac{\alpha \dot f}{t_0}.
\end{equation}
By the Cauchy-Schwarz inequality,
\begin{equation}
\label{0815027}
\frac{1}{2n}(\eta^{ij}f_i f_{j})^2 -(\frac{1}{t_0} +C)\eta^{ij}f_i f_{j}\leqslant C +\frac{C}{t^2_0} + \frac{\dot f^2}{4}.
\end{equation}
Therefore, at $(x_0,t_0)$, we have
\begin{equation}\label{0815028}
\eta^{ij} f_i f_{j} \leqslant C+\frac{C}{t_0}-\frac{\dot f}{2},
\end{equation}
where $C$ is a positive constant depending only on $\alpha$, $n$, $g$, $\hat\chi$ and $\|u\|_{C^4(X)}$.

On the other hand, by inequality \eqref{0815019}, at $(x_0,t_0)$, 
\begin{equation}
\label{0815029}
\frac{t_0}{2n}\dot f^2 +\alpha \dot f\leqslant C t_0 \eta^{ij} f_i f_{j} +C t_0 +\eta^{ij} f_i f_{ j},
\end{equation}
that is,
\begin{equation}
\label{0815030}
\frac{1}{2n} \dot f^2 +\frac{\alpha}{t_0} \dot f\leqslant C \eta^{ij} f_i f_{j} +C +\frac{1}{t_0}\eta^{ij} f_i f_{j}.
\end{equation}
By the Cauchy-Schwarz inequality,
\begin{equation}\label{0815031}
\frac{1}{2n} \dot f^2 +\frac{\alpha }{t_0}\dot f\leqslant \frac{1}{4} \left( \eta^{ij} f_i f_{j}\right)^2 +\frac{C}{t^2_0} +C.
\end{equation}
Hence at $(x_0,t_0)$, there exist a positive constant $C$ depending only on $\alpha$, $n$, $g$, $\hat\chi$ and $\|u\|_{C^4(X)}$ such that
\begin{equation}\label{0815032}
-\dot f\leqslant \frac{C}{t_0}+\frac{\eta^{ij}f_i f_{j}}{2}+C.
\end{equation}
Then inequalities \eqref{0815028} and \eqref{0815032} imply that
\begin{equation}\label{0815033}
  \eta^{ij} f_i f_{j}\leqslant C+\frac{C}{t_0} +\frac{\eta^{ij} f_i f_{j}}{4},
\end{equation}
that is,
\begin{equation}\label{0815034}
\eta^{ij}f_i f_{j}\leqslant C+\frac{C}{t_0}.
\end{equation}
Combining this inequality with \eqref{0815032}, we have
\begin{equation}\label{0815035}
-\dot f\leqslant C+\frac{C}{t_0}.
\end{equation}
Therefore,  there exist a positive constant $C$ depending only on $\alpha$, $n$, $g$, $\hat\chi$ and $\|u\|_{C^4(X)}$ such that at $(x_0,t_0)$,
\begin{equation}\label{0815036}
\eta^{ij}f_i f_{j} -\alpha \dot f\leqslant C+\frac{C}{t_0}.
\end{equation}
Then same arguments as in the former case imply inequality  \eqref{0815018}.
\end{proof}

By using Lemma \ref{0815017}, we prove the following Harnack-type inequality along the generalized Lagrangian mean curvature flow $(\ref{GLMCF1})$.
\begin{thm}\label{0815037}
There exists a positive constant $C$ depending only on $n$, $g$, $\hat\chi$ and $\|u\|_{C^4(X)}$ such that for any $0<t_1<t_2$, we have 
\begin{equation}\label{0815038}
\sup_{x\in X} v(x,t_1) \leqslant \inf_{x\in X }  v(x,t_2)\left(\frac{t_2}{t_1}\right)^{C} e^{\frac{C}{t_2-t_1}+C(t_2-t_1)}.
\end{equation}
\end{thm}
\begin{proof}
Let $x,y\in X$ be two arbitrary points and $\gamma$ be a geodesic with respect to metric $g$ such that
\begin{equation}\label{0815039}
\gamma(0)=x\ \ \ \text{and}\ \ \  \gamma(1)=y.
\end{equation}
We define curve $\xi(s):[0,1]\to X\times [t_1,t_2]$ by
\begin{equation}\label{0815040}
\xi(s)=(\gamma(s),(1-s)t_1+st_2),
\end{equation}
that is, $\gamma$ is a curve in $X\times [t_1,t_2]$ connecting $(x,t_1)$ and $(y,t_2)$. Then by Lemma \ref{0815017},
\begin{equation}\label{0815041}
\begin{split}
\log\frac{ v(x,t_1)}{ v(y,t_2)}=& -\int_0^1 \frac{\partial }{\partial s} f(\xi(s))ds\\
=& \int_0^1 (-df\dot \gamma -\dot f(t_2-t_1)) ds\\
\leqslant & \int_0^1 (\sqrt{\eta^{ij} f_i f_{j}}-\frac{t_2-t_1}{\alpha} \eta^{ij}f_if_{ j} -\dot f(t_2-t_1) +\frac{t_2-t_1}{\alpha} \eta^{ij}f_if_{j}) ds\\
\leqslant &\int_0^1 \left(\frac{\alpha}{4(t_2-t_1)} +C(t_2-t_1)+ \frac{C(t_2-t_1)}{(1-s)t_1+ st_2}\right)ds\\
=&\frac{C}{t_2-t_1} +C(t_2-t_1) +C\log\frac{t_2}{t_1},
\end{split}
\end{equation}
which implies that
\begin{equation}\label{0815042}
v(x,t_1)\leqslant v(y,t_2)\left(\frac{t_2}{t_1}\right)^{C} e^{\frac{C}{t_2-t_1} +C(t_2-t_1)},
\end{equation}
where $C$ is a positive constant depending only on $\alpha$, $n$, $g$, $\hat\chi$ and $\|u\|_{C^4(X)}$. Since $x,y$ are arbitrary points in $X$, after fixing some $\alpha\in(1,2)$, we complete the proof.
\end{proof}

\subsection{Exponential Convergence}
As an application of the Harnack-type inequality \eqref{0815038}, by following Cao's arguments \cite{Cao} for K\"ahler-Ricci flow, we prove the following estimates for
\begin{equation}\label{0815043}
\tilde u(t)= u(t)-\frac{\int_X u(t) dV_g}{\int_X dV_g}.
\end{equation}
\begin{thm}\label{08150441}
There exist positive constants $C_1$ and $C_2$ depending only on $n$, $g$, $\hat\chi$ and $\|u\|_{C^4(X)}$ such that
\begin{equation}\label{0815044}
\left|\frac{\partial\tilde u(t)}{\partial t}\right|\leqslant C_1 e^{-C_2 t}.
\end{equation}
\end{thm}
\begin{proof}
We denote $\varphi(t)$ and $\tilde{\varphi}(t)$ to be $\dot u(t)$ and $\dot{\tilde u}(t)$ respectively. It is easy to see that $\tilde{\varphi}(t)$ and $\varphi(t)$ satisfy
\begin{equation}\label{0815045}
\int_X \tilde\varphi(t) dV_g=0 \text{ and }  \frac{\partial \varphi(t)}{\partial t}=\eta^{ij}\varphi_{ij}.
\end{equation}
For any $t>0$ and $x,y\in X$, we have
\begin{equation}\label{0815046}
|\tilde \varphi(x,t)-\tilde\varphi(y,t)|=|\varphi(x,t)-\varphi(y,t)|.
\end{equation}
By the maximum principle, for any $0<t_1<t_2$, we have
\begin{equation}\label{0815047}
\sup_{y\in X} \varphi(y,t_2)\leqslant \sup_{y\in X} \varphi(y,t_1)\leqslant \sup_{y\in X}\varphi(y,0),
\end{equation} 
and  
\begin{equation}\label{0815048}
\inf_{y\in X}\varphi(y,t_2)\geqslant \inf_{y\in X} \varphi(y,t_1)\geqslant \inf_{y\in X}\varphi(y,0).
\end{equation}
Let $m$ be an arbitrary positive integer. For any $(x,t)$, we define
\begin{equation}\label{0815049}
\xi_m(x,t)=\sup_{y\in X} \varphi(y,m-1) -\varphi(x,m-1+t)
\end{equation}
and
\begin{equation}\label{0815050}
\psi_m(x,t)=\varphi(x,m-1+t)-\inf_{y\in X} \varphi(y,m-1).
\end{equation}
Then by \eqref{0815047} and \eqref{0815048}, both $\xi_m$ and $\varphi_m$ are non-negative and satisfy equations
\begin{equation}\label{0815051}
\frac{\partial \xi_m}{\partial t}(x,t) =\eta^{ij}(x,m-1+t) (\xi_m)_{i j}(x,t)
\end{equation}
and
\begin{equation}\label{0815052}
\frac{\partial \psi_m}{\partial t}(x,t) =\eta^{ij}(x,m-1+t) (\psi_m)_{i j}(x,t).
\end{equation}

If $\varphi(x,m-1)$ is constant, then $\varphi(x,t)$ must be constant for all $t\geqslant m-1$ by the maximum principle. Then \eqref{0815045} and \eqref{0815046} imply that $\tilde\varphi(x,t)=0$ on $X\times[m-1,+\infty)$. Hence \eqref{0815044} is obvious.

Nexy, we assume that $\varphi(x,m-1)$ is not constant. It follows that at $t=0$, $\xi_m$ must be positive at some point $x_0$. By the strong maximum principle, $\xi_m(x,t)$ must be positive on $X\times (0,+\infty)$. Similarly, we also have $\psi_m(x,t)>0$ on $X\times (0,+\infty)$. Then applying Theorem \ref{0815037} with $t_1=\frac{1}{2}$ and $t_2=1$, we obtain
\begin{equation}\label{0815053}
\begin{split}
\sup_{y\in X} \varphi(y,m-1) -\inf_{y\in X} \varphi(y,m-\frac{1}{2})&\leqslant C(\sup_{y\in X} \varphi(y,m-1) -\sup_{y\in X}\varphi(y,m)),\\
\sup_{y\in X} \varphi(y,m-\frac{1}{2}) -\inf_{y\in X} \varphi(y,m-1)&\leqslant C(\inf_{y\in X} \varphi(y,m) -\inf_{y\in X} \varphi(y,m-1)),
\end{split}
\end{equation}
where $C>1$ is a constant depending only on $n$, $g$, $\hat\chi$ and $\|u\|_{C^4(X)}$. Denote
\begin{equation}\label{0815054}
\chi(t)=\sup_{y\in X} \varphi(y,t)- \inf_{y\in X} \varphi(y,t).
\end{equation}
By \eqref{0815053}, we have
\begin{equation}\label{0815055}
\chi(m-1)+\chi(m-\frac{1}{2})\leqslant C(\chi(m-1)-\chi(m)).
\end{equation}
Since $\chi$ is a non-negative, there holds
\begin{equation}\label{0815056}
\chi(m)\leqslant \frac{C-1}{C}\chi(m-1).
\end{equation}
By induction, we obtain
\begin{equation}\label{0815057}
\chi(m)\leqslant \left(\frac{C-1}{C}\right)^m \chi(0).
\end{equation}
According to inequalities \eqref{0815047} and \eqref{0815048}, $\chi(t)$ is decreasing in $t$. Therefore, for any $t\in[m,m+1)$,
\begin{equation}\label{0815058}
\chi(t)\leqslant \chi(m)\leqslant \left(\frac{C-1}{C}\right)^m \chi(0)\leqslant \left(\frac{C-1}{C}\right)^t\left(\frac{C}{C-1}\right)^{t-m} \chi(0)\leqslant C_1e^{-C_2t},
\end{equation}
where $C_1=\frac{C\chi(0)}{C-1}$ and $C_2=\log \frac{C}{C-1}$. Since $m$ is arbitrary, we have
\begin{equation}\label{0815058}
\chi(t)\leqslant C_1e^{-C_2t},
\end{equation}
where $C_1$ and $C_2$ are positive constants depending only on $n$, $g$, $\hat\chi$ and $\|u\|_{C^4(X)}$. 

Since $\int_X \tilde\varphi dV_g=0$, for any $t>0$, there must be a point $x_t\in M$ such that $\tilde\varphi(x_t,t)=0$. Then for any $(x,t)\in M\times [0,\infty)$, we have
\begin{equation}\label{0815059}
|\tilde\varphi(x,t)|=|\tilde\varphi(x,t)-\tilde\varphi(x_t,t)|=|\varphi(x,t)-\varphi(x_t,t)|\leqslant \chi(t)\leqslant C_1 e^{-C_2t}.
\end{equation}
We complete the proof.
\end{proof}
At last, we prove the following convergence result.
\begin{thm}\label{theo5.5}
There exists a constant $\delta_0>0$ such that if $\rho\leqslant \delta_0$ at $t=0$, then $\hat \chi_{u(t)}$ converges exponentially to $\hat \chi$ along the generalized Lagrangian mean curvature flow $(\ref{GLMCF1})$ as $t$ goes to $+\infty$.
\end{thm}
\begin{proof}
By using Theorem \ref{08150441}, for any $0<t_1<t_2$, we have
\begin{equation}\label{0815060}
|\tilde u(x,t_1)-\tilde u(x,t_2)|\leqslant \int_{t_1}^{t_2}\left|\frac{\partial}{\partial t}\tilde u(x,t)\right|dt\leqslant \int_{t_1}^{+\infty}\left|\frac{\partial}{\partial t}\tilde u(x,t)\right|dt\leqslant \frac{C_1}{C_2}e^{-C_2t_1},
\end{equation}
which implies that $\{\tilde u(t)\}$ is a Cauchy sequence in $C^0$-sense with respect to $t$.

According to Theorem \ref{0824031} and Theorem \ref{0824061}, $v(t)= u(t)-\theta \left( \hat \chi \right)t $ is uniformly $C^\infty$-bounded along the flow $(\ref{GLMCF1})$, so there exists a subsequence $v(t_i)= u(t_i)-\theta \left( \hat\chi \right) t_i$ converging to a smooth function $v_\infty$ on $X$ and $\hat\chi_{v(t_i)}=\hat\chi+dv(t_i)$ converging to $\hat\chi_{v_\infty}=\hat\chi+dv_\infty$ as $t_i$ goes to $+\infty$. Since $\theta(\hat\chi_{v(t)})=\theta(\hat\chi_{u(t)})=\dot u(t)$, inequality \eqref{0815058} implies that $\theta(\hat\chi_{v(t_i)})$ converges to a constant in $C^\infty$-sense. Hence $\hat\chi_{v_\infty}$ induces a special Lagrangian graph in $T^*X$. Furthermore, we conclude that $\hat\chi_{v_\infty}=\hat\chi$ by using the uniqueness theorem (Theorem \ref{0815100}), which implies that $v_\infty$ is a constant. By the definition \eqref{0815043},  $\tilde u(t_i)$ must converge to $0$ in $C^\infty$-sense. Since $\{\tilde u(t)\}$ is a Cauchy sequence in $C^0$-sense, $\tilde u(t)$ converges to $0$ in $C^0$-sense as $t$ goes to $+\infty$. Let $t_2$ goes to $+\infty$ in \eqref{0815060}, we have 
\begin{equation}\label{0815061}
\|\tilde u(t)\|_{C^0(X)}\leqslant \frac{C_1}{C_2}e^{-C_2t}\ \ \ \text{on}\ \ \ [0,+\infty).
\end{equation}
Hence $\tilde u(t)$ converges exponentially to $0$ in $C^0$-sense.

We claim that $\tilde u(t)$ actually converges to $0$ in $C^\infty$-sense as $t$ goes to $+\infty$. If not, there exist $\varepsilon_0>0$, $k_0\in\mathbb{N}^+$ and a time sequence $t'_j$ converging to $+\infty$ such that
\begin{equation}\label{0906005}
\|\tilde u(t'_j)\|_{C^{k_0}(X)}\geqslant\varepsilon_0.
\end{equation}
But repeating above arguments for $v(t'_j)$, we conclude that there is a subsequence also denoted by $v(t'_j)$ converging to a constant and hence $\tilde u(t'_j)$ converges to $0$ in $C^\infty$-sense. This is a contradiction with \eqref{0906005}. We prove the claim. 

At last, we prove that the smooth convergence should be exponentially fast. For any $k\in\mathbb{N^+}$, by using \eqref{0815044} and integration by parts,
\begin{equation}\label{0815053}
\begin{split}
\int_X \left| D^k\tilde u(t)\right|^2 dV_g&=-\int_t^\infty\frac{\partial}{\partial s}\int_X \left| D^k\tilde u(s)\right|^2_g dV_g\ ds\\
&=-2\int_t^\infty\int_X D^k\tilde u(s)\ast_g D^k\frac{\partial}{\partial s}\tilde u(s) dV_g\ ds\\
&\leqslant2\int_t^\infty\int_X \left|D^{2k}\tilde u(s)\right|_g \left|\frac{\partial}{\partial s}\tilde u(s)\right| dV_g\ ds\\
&\leqslant2\int_t^\infty\left(\int_X \left|D^{2k}\tilde u(s)\right|^2_gdV_g \right)^{\frac{1}{2}} \left(\int_X\left|\frac{\partial}{\partial s}\tilde u(s)\right|^2 dV_g\right)^{\frac{1}{2}} ds\\
&\leqslant Ce^{-C_2t}.
\end{split}
\end{equation}
Hence $\|\tilde u(t)\|_{W^{k,2}}$ converges exponentially  to $0$ as $t$ goes to $+\infty$. Then by the Sobolev embedding theorem, we conclude that $\tilde u(t)$ converges exponentially  to $0$ in $C^\infty$-sense and hence $\hat \chi_{u(t)}$ converges exponentially  to $\hat\chi$ in $C^\infty$-sense as $t$ goes to $+\infty$.
\end{proof}
{\it Proof of Theorem \ref{0825003}.}\ \ Theorem \ref{0825003} follows from Theorem \ref{0824061} and Theorem \ref{theo5.5} directly. \QEDB

\end{document}